\documentclass{amsart}

\usepackage{amsmath,amssymb,amsfonts,amssymb,amsthm}
\usepackage{verbatim}
\usepackage[usenames]{color}
\usepackage{hyperref}
\usepackage{url}
\usepackage{tikz,tikz-qtree,ifthen,cancel}
\usepackage{array,tikz-qtree,ifthen,cancel}
\usepackage{graphicx}
\usepackage{adjustbox}
\usepackage{amsthm,graphicx,tikz,appendix,tikz-qtree,ifthen,cancel}
\usetikzlibrary{calc,shapes,patterns,positioning}

\newtheorem{thm}{Theorem}[section]
\newtheorem{lem}[thm]{Lemma}
\newtheorem{cor}[thm]{Corollary}
\newtheorem{observation}[thm]{Observation}

\theoremstyle{remark}
\newtheorem{rem}[thm]{Remark}

\theoremstyle{definition}
\newtheorem{defn}[thm]{Definition}

\theoremstyle{remark}

\newcommand{\POC}{Parallel $1$'s Criterion}
\newcommand{\SPOC}{Strict Parallel $1$'s Criterion}

\newcommand{\PLSST}{pair-linked strict similarity type}

\newcommand{\om}{\omega}

\newcommand{\sse}{\subseteq}
\newcommand{\contains}{\supseteq}

\DeclareMathOperator{\Simss}{Sim^{ss}}
\DeclareMathOperator{\Simp}{Sim^{ep}}
\DeclareMathOperator{\Sim}{Sim}

\DeclareMathOperator{\cl}{cl}
\DeclareMathOperator{\BL}{BL}

\newcommand{\re}{\restriction}

\newcommand{\bC}{\mathbb{C}}
\newcommand{\bD}{\mathbb{D}}

\newcommand{\bS}{\mathbb{S}}

\newcommand{\E}{\mathrm{E}}

\newcommand{\G}{\mathrm{G}}

\newcommand{\bG}{\mathbf{G}}

\newcommand{\bfA}{\mathbf{A}}

\newcommand{\psim}{\stackrel{ep}{\sim}}

\newcommand{\sssim}{\stackrel{ss}{\sim}}

\newcommand{\ra}{\rightarrow}

\newcommand{\lgl}{\langle}
\newcommand{\rgl}{\rangle}

\newcommand{\Fraisse}{Fra{\"{i}}ss{\'{e}}}
\newcommand{\Hubicka}{Hubi{\v{c}}ka}
\newcommand{\Komjath}{Komj{\'{a}}th}

\newcommand{\Masulovic}{Ma{\v{s}}ulovi{\'{c}}}

\newcommand{\Rodl}{R{\"{o}}dl}

\newcommand{\noprint}[1]{\relax}

\title[Exact big Ramsey degrees of the triangle-free Henson graph]{The Ramsey theory  of  the universal\\ homogeneous triangle-free graph Part II:\\  Exact big Ramsey degrees}
\author{Natasha Dobrinen}
\address{University of Denver\\
Department of Mathematics, 2390 S. York St., Denver, CO 80208, USA}
\email{natasha.dobrinen@du.edu}
  \urladdr{\url{http://cs.du.edu/~ndobrine}}

\thanks{The  author gratefully acknowledges support from   National Science Foundation Grant DMS-1901753}



\begin{document}

\maketitle

\begin{abstract}
Building on work in \cite{DobrinenJML20},
for each finite triangle-free graph $\bG$, we
determine the equivalence relation on the
 copies of $\bG$ inside the
  universal homogeneous triangle-free  graph, $\mathcal{H}_3$,  with the smallest number of equivalence classes so that each one of the classes  persists in every isomorphic subcopy of $\mathcal{H}_3$.
This characterizes the exact big Ramsey degrees of $\mathcal{H}_3$.
It follows that the triangle-free Henson graph is a big Ramsey structure.
\end{abstract}

%%%%%%%%%%%%%%%%%%%%%%%%
%%%%%%%%%%%%%%%%%%%%%%%%
%%%%%%%%%%%%%%%%%%%%%%%%
%%%%%%%%%%%%%%%%%%%%%%%%

\section{Overview}\label{sec.overview}

This paper is a sequel to \cite{DobrinenJML20},  in which  the author proved  that the triangle-free Henson graph has finite big Ramsey degrees.
The original hope for that paper was to find the exact degrees, and we conjectured that the bounds found were optimal.
In this paper, we show that  while those bounds were correct for singletons, edges, and  non-edges,
more generally they were  not exact.
However, the structure of the  coding trees developed in that paper did achieve the best bounds known so far,
and a small  but significant modification of those coding trees in this paper will enable us to  prove the exact bounds.
In Section \ref{sec.NewSquiggleTree}, we
improve the results of Section 7 in  \cite{DobrinenJML20} and  then apply theorems  from
\cite{DobrinenJML20} to obtain better upper bounds.
In Section \ref{sec.brd},
  we prove   that these bounds are exact.
This is the first result on exact big Ramsey degrees for structures with forbidden irreducible substructures.

%Say something about embeddings versus copies and application of Zucker's work.

%Say something about Zucker and {\Hubicka}'s work.

%%%%%%%%%%%%%%%%%%%%%%%%
%%%%%%%%%%%%%%%%%%%%%%%%
%%%%%%%%%%%%%%%%%%%%%%%%
%%%%%%%%%%%%%%%%%%%%%%%%

\section{Introduction}\label{sec.intro}

The  universal homogeneous triangle-free  graph,  denoted by $\mathcal{H}_3$ and also known as the {\em triangle-free Henson graph},
is the \Fraisse\ limit of the class of finite triangle-free graphs, $\mathcal{G}_3$.
We say that $\mathcal{H}_3$ has {\em finite big Ramsey degrees} if
for each finite triangle-free graph $\bG\in\mathcal{G}_3$,
 there is  a positive integer  $T$
 such that
 $$
 \mathcal{H}_3\ra(\mathcal{H}_3)^{\bG}_{k,T}
 $$
holds for any $k\ge 1$.
When  such a $T$ exists,
we let $T(\bG,\mathcal{H}_3)$
denote the least such $T$, and call it the {\em big Ramsey degree} of $\bG$ in $\mathcal{H}_3$, using the terminology in \cite{Kechris/Pestov/Todorcevic05}.

The reader interested in a broader understanding of  the area is  referred  to the papers \cite{DobrinenJML20} and \cite{DobrinenH_k19}, where   extensive  background on big Ramsey degrees of the Henson graphs, including results known at the time, is provided.
An expository article  \cite{DobrinenIfCoLog20} on forcing and the method of coding trees also provides a general overview of this area.
Here,  we shall provide a minimal overview of the problem, an update on currently known results,  and the results in this paper.

It is important
to distinguish between proving  that
big Ramsey degrees are finite
(that is, finding upper bounds)
and characterizing the actual numbers  $T(\bG,\mathcal{H}_3)$ via some structures from which they can be calculated.
In the first case, we say that $\mathcal{H}_3$ has  {\em finite big Ramsey degrees}
and in the second case, we say that
{\em exact big Ramsey degrees} are  {\em characterized}.
When exact big Ramsey degrees are characterized using some extra structure (in our case, some structure implicit in coding trees),
then the results of Zucker in \cite{Zucker19} regarding big Ramsey structures and universal
completion flows in topological dynamics apply.
This is one reason why characterizing exact big Ramsey degrees is of much current interest.

%Lastly,  we say that the  big Ramsey degrees are   {\em calculated}   if the number $T(\bG,\mathcal{H}_3)$ has been calculated, or found in terms of some recursive function from which it can be easily calculated.  Big Ramsey degrees  so far have only  been calculated for the rationals by Devlin in \cite{DevlinThesis},  the structures $\bQ_n$  and the  tournament $\bS(2)$ by Laflamme, Nguyen Van Th\'{e}, and Sauer in  \cite{Laflamme/NVT/Sauer10}, and the  circular directed graphs $\mathbf{S}(n)$ for $n\ge 3$  in a recent paper of Barbosa \cite{Barbosa20}. For the Rado graph,  big Ramsey degrees   have been calculated up to the cardinality of the graphby Larson in  \cite{Larson08}.

%***
%Add the results of NVT for ultrametric space:  - did he calculate the big Ramsey degrees?
%***

In many  cases where exact big Ramsey degrees have been characterized,
this has been done via  finding canonical partitions via some sort of tree structures.
(A new method using category theory has recently  been  successfully developed by
\Masulovic\ in \cite{Masulovic18}  and Barbosa in \cite{Barbosa20}.)
In the terminology of  \cite{Laflamme/Sauer/Vuksanovic06},
a partition $\mathcal{P}_0,\dots,\mathcal{P}_{T-1}$ of the copies of a structure $\bfA$ in an infinite structure $\mathcal{S}$ is {\em canonical}
if the following holds:
For each coloring
$f:{\mathcal{S}\choose \bfA}\ra k$, where $k\ge 2$,
 there is an isomorphic substructure $\mathcal{S}'$ of $\mathcal{S}$
such that for each $n<T$,
$f$ takes one color in ${\mathcal{S}'\choose \bfA}\cap \mathcal{P}_n$;
 moreover,
for each isomorphic substructure $\mathcal{S}'$ of $\mathcal{S}$,
${\mathcal{S}'\choose \bfA}\cap \mathcal{P}_n$ is non-empty.
This latter property is called {\em persistence}.

In \cite{DobrinenJML20},
we developed the notion of strong coding tree and  of  incremental strict similarity type
%(Definition \ref),
 and proved that the number of incremental strict similarity types of antichains coding $\bG$ is an upper bound for $T(\bG,\mathcal{H}_3)$.
%The following is a paraphrase of Theorem 8.9 in  \cite{DobrinenJML20} (upon application of the antichain $\bD$ from Section 9 of that paper).
%\begin{thm}[\cite{DobrinenJML20}]\label{thm.JMLUpperBounds}
%Let $\bG$ be a finite triangle-free graphs, and let $A$ be a finite antichain of coding nodes in a strong coding tree $T$. Given a finite coloring  $h$ of all subsets of $T$ which are strictly similar to $A$,  there is an incremental strong coding tree $S\le T$ such that all subsets of $S$ strictly similar to $A$ have the same $h$ color. It follows that the big Ramsey degree $T(\bG,\mathcal{H}_3)$ is bounded by the number of strict similarity types of antichains coding $\bG$.
%\end{thm}
In this paper, we shall
refine work from Sections 7--9 of \cite{DobrinenJML20}  to prove better upper bounds.
Then we shall  use  ideas from the proof of Theorem 4.1 in \cite{Laflamme/Sauer/Vuksanovic06}
and build some new methods for triangle-free graphs to prove that these better upper bounds are exact.

 The characterization of the big Ramsey degrees of the triangle-free Henson graph  is via the notion of essential pair similarity.
 Throughout, we work with an enumerated copy of $\mathcal{H}_3$, and its induced coding tree $\bS$.
 An {\em essential linked pair}
 is a pair of nodes $s,t\in\bS$ such that
 $s$ and $t$ both code an edge with a common vertex of $\mathcal{H}_3$,
 but the least vertex of $\mathcal{H}_3$ with which $s$ codes an edge differs from the least vertex of $\mathcal{H}_3$ with which $s$ codes an edge.
 Two trees with coding nodes  $A$ and $B$ are {\em essential pair similar} (or {\em ep-similar}) if  $A$ and $B$ are strongly similar, and additionally the  strong similarity map from $A$ to $B$ preserves the order in which new essential linked pairs appear.
 Given a finite triangle-free graph $\bG$,
 we let $\Simp(\bG)$ be a set of representatives
 from among the ep-similarity types of
   antichains $A$ (of coding nodes)
  representing $\bG$ such that each pair of coding nodes $c^A_m,c^A_n$ in $A$ ($m<n$)
   coding a non-edge between their represented vertices is linked (meaning they code an edge with a common vertex in $\mathcal{H}_3$).
 \vskip.1in

\noindent{\bf Theorem \ref{thm.canonpart}.}
Let $\bG$ be a finite triangle-free graph and let $h$ be a coloring of all copies of $\bG$ inside $\mathcal{H}_3$.
Then there is a subgraph $\mathcal{H}$ of $\mathcal{H}_3$ which isomorphic to $\mathcal{H}_3$
in which  for each $A\in\Simp(\bG)$,
all copies of $\bG$
represented by an antichain which is ep-similar to $A$
 have the same color.
Moreover,  each  ep-similarity type in $\Simp(\bG)$ persists in the coding tree induced by $\mathcal{H}$.
\vskip.1in

This characterizes the exact big Ramsey degrees of finite triangle-free graphs in the triangle-free Henson graph.
\vskip.1in

\noindent{\bf Theorem \ref{cor.main}.}
Given a finite triangle-free graph $\bG$,
the big Ramsey degree of $\bG$ in the triangle-free graph is exactly  the number of essential pair similarity types of strongly skew antichains coding $\bG$:
\begin{equation}
T(\bG,\mathcal{H}_3)=|\Simp(\bG)|.
\end{equation}

Some related  recent results
deserve mention.
Zucker proved in
 \cite{Zucker20}
 that \Fraisse\ classes with finitely many binary relations and finitely many forbidden irreducible substructures have finite big Ramsey degrees.
 This generalized the work of the author
  in \cite{DobrinenJML20} and \cite{DobrinenH_k19}
  showing that
  all $K_n$-free Henson graphs  have finite big Ramsey degrees.
Very recently, \Hubicka\ has found the first non-forcing proof that $\mathcal{H}_3$ has finite big Ramsey degrees in \cite{Hubicka_CS20}.
In that paper,  he used
  the Carlson-Simpson theorem to prove that the universal homogeneous partial order has finite big Ramsey degrees, and by similar methods was able to recover bounds for the big Ramsey degrees of $\mathcal{H}_3$.
 The proofs
 in  \cite{Zucker20}  and \cite{Hubicka_CS20}
  that upper bounds exist
  are quite a bit shorter than those of the author in \cite{DobrinenJML20} and \cite{DobrinenH_k19}.
This comes at the expense of looser upper bounds.
Part of the motivation of the extreme  structure of the coding trees of the author was to
construct a space of coding trees which themselves
recover indivisibility results of \Komjath\ and \Rodl\ for $\mathcal{H}_3$ \cite{Komjath/Rodl86} and of El-Zahar and Sauer  for the rest of the $K_n$-free Henson graphs \cite{El-Zahar/Sauer89}.
The other part of the motivation was to prove exact big Ramsey degrees.
This paper shows that adding a small but important
requirement to the coding trees in \cite{DobrinenJML20} results in the exact big Ramsey degrees for $\mathcal{H}_3$.
Something similar ought to be possible for  other binary relational structures with some forbidden irreducible substructures.

%More recently, \Hubicka\ has used the Carlson-Simpson theorem to
%\cite{Hubicka20}.

%In particular, neither of those papers achieves the bounds proved in \cite{DobrinenJML20}.

%***
%Also, doesn't my work automatically give bounds for the ordered versions because of the lex ordering in the trees?
%***

%Discuss the embedding formulation of big Ramsey degrees, the back and forth between the two, and how our results show that these \Fraisse\ limits are big Ramsey structures, so Andy's work applies.

%%%%%%%%%%%%%%%%%%%%%%
%%%%%%%%%%%%%%%%%%%%%%
%%%%%%%%%%%%%%%%%%%%%%
%%%%%%%%%%%%%%%%%%%%%%

\section{Review}\label{sec.Review}

In this section, we review  some concepts   from  \cite{DobrinenJML20}  to ease the reading of this paper.
The reader familiar with that paper can skip this section.

%%%%%%%%%%%%%%%%%%%%%%%%%
%%%%%%%%%%%%%%%%%%%%%%%%%

\subsection{The strong triangle-free   coding tree $\bS$}\label{subsec.bS}

The set $2^{<\om}$ is the collection of all finite length sequences of $0$'s and $1$'s.
We let $0^{<\om}$ denote $\{0\}^{<\om}$, the collection of all finite sequences of $0$'s.
Given $s\in 2^{<\om}$,  we let $|s|$ denote the {\em length} of $s$.
The {\em meet} of two nodes $s,t\in 2^{<\om}$, denoted $s\wedge t$,
is the longest member $u\in 2^{<\om}$ which is an initial segment of both $s$ and $t$.
In particular, if $s\sse t$ then $s\wedge t=s$.
A set  of nodes $A\sse 2^{<\om}$ is {\em closed under meets}
if $s\wedge t$ is in $A$, for each pair $s,t\in A$.
Given $A\sse 2^{<\om}$, we let $\cl(A)$ denote the
set $\{s\wedge t:s,t\in A\}$
and call this the
{\em meet-closure} of $A$.
Since $s$ and $t$ are allowed to be equal in the definition of $\cl(A)$, the meet-closure of $A$ contains $A$.
We adhere to the  following definition of tree, which is standard for  Ramsey theory on trees.

\begin{defn}\label{defn.tree}
A  subset $T\sse 2^{<\om}$  is a {\em tree}
if $T$
 is closed under meets and  for each pair $s,t\in T$
with $|s|\le |t|$,
$t\re |s|$ is also in $T$.
\end{defn}

Graphs can be coded via nodes in  $2^{<\om}$ with the edge relation coded via passing number.
Given two vertices $v,w$ in some graph $\mathrm{G}$,
two nodes
 $s,t\in 2^{<\om}$   {\em represent} $v$ and $w$
  if, assuming $|s|<|t|$,
then
$v$ and $w$ have an edge between them
if and only if
 $t(|s|)=1$.
The number $t(|s|)$ is called the {\em passing number} of $t$ at $s$ (see \cite{Sauer06} for the first usage of this terminology).
The following appears as Definition 3.1 in \cite{DobrinenJML20}.

\begin{defn}[Tree with coding nodes]\label{defn.treewcodingnodes}
A {\em tree with coding nodes}
is a structure $(T,N;\sse,<,c^T)$ in the language of
$\mathcal{L}=\{\sse,<,c\}$,
 where $\sse$ and $<$ are binary relation symbols  and $c^T$ is a unary function symbol,
 satisfying the following:
 $T$ is a subset of  $2^{<\om}$ satisfying that   $(T,\sse)$ is a tree, $N\le \om$ and $<$ is the usual linear order on $N$,  and $c^T:N\ra T$ is an injective  function such that  $m<n<N$ implies $|c^T(m)|<|c^T(n)|$.
\end{defn}

We often denote $c^T(n)$ by $c^T_n$, especially when working with more than one tree at the same time.
The following is Definition 3.3 in \cite{DobrinenJML20}.

\begin{defn}\label{def.rep}
A graph $\G$ with vertices enumerated as $\lgl v_n:n<N\rgl$ is {\em  represented}  by a tree $T$ with  coding nodes $\lgl c^T_n:n<N\rgl$
if and only if
for each pair $i<n<N$,
 $v_n\, \E\, v_i\longleftrightarrow  c^T_n(l_i)=1$.
We will often simply say that $T$ {\em codes} $\G$.
\end{defn}

The triangle-free Henson graph $\mathcal{H}_3$ can be represented by a tree with coding nodes.
\vskip.1in

\noindent {\bf Construction of the strong triangle-free coding tree $\bS$}.
Let $\mathcal{H}_3$ be a Henson graph with universe
 $\lgl v_n:n<\om\rgl$ labeled in order-type $\om$.
 Assume that this representation of $\mathcal{H}_3$ has the following properties:
\begin{enumerate}
 \item
For each $n<\om$, $v_n\, E\, v_{n+1}$.
 \item
 For each $n<\om$, for all $i\le 2n$, $v_{2n+1}\,E\, v_i$ if and only if $i=2n$.
\end{enumerate}

Let $\bS$ be the coding tree for $\mathcal{H}_3$ constructed as follows:
Let $c^{\bS}_0$ be the empty sequence; this coding node represents the vertex $v_0$.
In general, given $n>0$ and supposing $c^{\bS}_m$ is defined for all $m<n$,
 take  $c^{\bS}_n$ to be the (unique) node in $2^n$ such that for all $m<n$, $c^{\bS}_n$ has passing number $1$ at $c^{\bS}_m$ if and only if $v_n\, E\, v_m$.
 The $m$-th level of $\bS$ consists of all nodes $s\in 2^m$ for which there exists  an $n\ge m$ such that  $s\sse c^{\bS}_n$.
\vskip.1in

The reader  familiar with  the paper \cite{DobrinenJML20} will notice that this construction of  $\bS$  slightly differs from the  presentation   of  Example 3.15 there.
In \cite{DobrinenH_k19}, we streamlined the presentation of  $\bS$ and of strong coding trees, and that is reflected here.
Requirement (i) here is the same as in \cite{DobrinenJML20}, and requirement (ii) here
is the part  of (ii) in \cite{DobrinenJML20}
corresponding to $F_{3i+j}=\emptyset$ for all $i<\om$ and $j\in\{0,2\}$.
The rest of requirements (ii) and (iii) from
 \cite{DobrinenJML20} were  formulated to  ensure that the coding nodes would be dense in the tree and represent the Henson graph.
 This is taken care of by
 enumerating a copy of the Henson graph and using it to define the corresponding coding tree (see \cite{DobrinenH_k19}).

\begin{rem}
Our requirement (i)  ensures that all coding nodes in $\bS$ (besides $c^{\bS}_0$) do not split in $\bS$.
This, in addition to using skew trees (see Definition \ref{def.sct}), had the effect of  recovering directly, from our Ramsey theory of strong coding trees in \cite{DobrinenJML20}, the result of Komj\'{a}th and \Rodl\ \cite{Komjath/Rodl86} that $\mathcal{H}_3$ is indivisible.
%That result is not  directly recovered from  the Ramsey theory of the  coding trees  of  Zucker in \cite{Zucker20}, but rather will follow upon further work.
\end{rem}

%%%%%%%%%%%%%%%%%%%%%%
%%%%%%%%%%%%%%%%%%%%%%
%%%%%%%%%%%%%%%%%%%%%%

\subsection{Ramsey theorem finite trees with the  \SPOC}\label{subsec.RTReview}

In this subsection, we review the Ramsey theorem   from  \cite{DobrinenJML20}, which aids in providing upper bounds for the big Ramsey degrees in the triangle-free Henson graph.

Recalling Definition
\ref{defn.treewcodingnodes},
 let  $T\sse 2^{<\om}$  be a tree with coding nodes $\lgl c^T_n:n<N\rgl$, where $N\le \om$,
and let
  $\ell^T_n$ denote $|c^T_n|$.
  Let
$\widehat{T}$ denote  $\{t \re n:t\in T$ and $n\le |t|\}$, the tree of all initial segments   of members of $T$.
A node $s\in T$  is called  a {\em splitting node}  if both $s^{\frown}0$ and $s^{\frown}1$  are in $\widehat{T}$.
Given $t\in T$, the  {\em level} of  $T$ of length  $|t|$  is  the set of all $s\in T$ such that $|s|=|t|$.
 $T$   is {\em skew} if
each level of $T$ has exactly one of either a  coding node or a splitting node.
A skew tree
$T$ is {\em strongly skew} if  additionally
for each splitting node $s\in T$,
every
 $t\in T$ such that $|t|>|s|$ and  $t\not\supset s$ also
satisfies
$t(|s|)=0$.
Given a strongly skew  coding tree $T\sse\bS$,  let
 $\lgl d^T_m:m<M\rgl$  enumerate the  coding  and splitting nodes of $T$ in increasing order of length;
 the nodes  $d^T_m$ are called  the {\em critical nodes} of $T$.
Let $m_n$ denote the integer
such that $d^T_{m_n}=c^T_n$.
The {\em $0$-th interval} of $T$ is the set of those nodes in $T$ with lengths in $[0,\ell^T_0]$, and for $0<n<N$, the {\em $n$-th interval} of $T$ is the set of those nodes in  $T$ with lengths in $(\ell^T_{n-1},\ell^T_n]$.

The {\em lexicographic order} on $2^{<\om}$
between two nodes
$s,t\in 2^{<\om}$, with neither extending the other,
is defined by
$s<_{\mathrm{lex}} t$ if and only if $s\contains (s\wedge t)^{\frown}0$ and $t\contains (s\wedge t)^{\frown}1$.
It is important to note that if $T$ is a strongly skew subset of $\bS$, then each node $s$ at the level of a coding node $c_n$ in $T$
has exactly one  immediate extension in $\widehat{T}$.
For two nodes $s,t$ with $|s|<|t|$, the number $t(|s|)$ is called the {\em passing number} of $t$ at $s$.
The following appears as Definition 4.9 in \cite{DobrinenJML20},
 augmenting Sauer's Definition 3.1 in \cite{Sauer06}  to the setting of  trees with coding  nodes.

\begin{defn}\label{def.3.1.likeSauer}
Let $S,T$ be strongly skew meet-closed subsets of $\bS$.
The function $f:S\ra T$ is a {\em strong similarity} of $S$ to $T$ if for all nodes $s,t,u,v\in S$, the following hold:
\begin{enumerate}
\item
$f$ is a bijection.
\item
$f$ preserves lexicographic order: $s<_{\mathrm{lex}}t$ if and only if $f(s)<_{\mathrm{lex}}f(t)$.
\item
$f$ preserves initial segments:
$s\wedge t\sse u\wedge v$ if and only if $f(s)\wedge f(t)\sse f(u)\wedge f(v)$.

\item
$f$ preserves meets:
$f(s\wedge t)=f(s)\wedge f(t)$.

\item
$f$ preserves relative lengths:
$|s\wedge t|<|u\wedge v|$ if and only if
$|f(s)\wedge f(t)|<|f(u)\wedge f(v)|$.

\item $f$ preserves coding  nodes:
$f$ maps the set of  coding nodes in $S$
onto the set of coding nodes in $T$.
\item

$f$ preserves passing numbers at coding nodes:
If $c$ is a coding node in $S$ and $u$ is a node in $S$ with $|u|>|c|$,
then $f(u)(|f(c)|)=u(|c|)$;
in words, the passing number of  $f(u)$ at $f(c)$  equals the passing number of $u$ at $c$.
\end{enumerate}
\end{defn}

We are going to make one terminology shift in this paper to ease descriptions of a property central to the exact big Ramsey degrees of triangle-free graphs.

\begin{defn}[Linked Pairs]\label{defn.linked}
We shall call a pair of nodes $\{s,t\}\sse\bS$ {\em linked} if and only if  there is some $\ell<\min(|s|,|t|)$ such that $s(\ell)=t(\ell)=1$.
We say that $s$ and $t$ are {\em linked  at level $\ell$} if and only if $s(\ell)=t(\ell)=1$.
In the special case that  $\{s,t\}$ is a linked set such that $s\wedge t$ is not in $0^{<\om}$, then we say that $s$ and $t$ are {\em base-linked}.
Given a subtree $A\sse \bS$ and a node $s\in A$,
let $\BL_A(s)$ denote the collection of all nodes $t\in A\re (|s|+1)$ which are base-linked with $s$.

A level subset  $X$ of $\bS$ is {\em pairwise linked}
if and only if   each  pair  of nodes $\{s,t\}\sse X$
is linked.
Given a  subtree $A\sse \bS$ and
 we say that a level  subset $X\sse A$ is a {\em maximal pairwise linked set} in $A$,
  or is {\em maximally linked}, if and only if
 $X$ is  pairwise linked, and  for  any $t\in A\re \ell_X$ not in $X$,
 the set $X\cup\{t\}$ is not pairwise  linked.
\end{defn}

\begin{rem}
What we are calling {\em linked} in this paper is exactly what we called {\em parallel $1$'s} in \cite{DobrinenJML20} and
  a {\em pre-$3$-clique} in \cite{DobrinenH_k19}.
It will be much less cumbersome to use this new terminology here.
\end{rem}

Note that a pair of nodes  $s$ and $t$ are linked at level $\ell$ if and only if for any $m$ and $n$ such that the  coding nodes $c_m$ and $c_n$ in $\bS$ extend $s$ and $t$, respectively,
the vertices $v_m$ and $v_n$ in the enumerated  Henson graph $\mathcal{H}_3$
both have an edge with the vertex $v_{\ell}$.
The first instances where pairs of nodes code an edge with a common vertex in $\mathcal{H}_3$  will turn out to be central to the exact big Ramsey degrees of triangle-free graphs.

Given a subset $A\sse\bS$ and  $\ell<\om$, define
\begin{equation}\label{eq.A_1}
A_{\ell,1}=\{s\re (\ell+1):s\in A,\ |s|\ge \ell+1,\mathrm{\  and\ }s(\ell)=1\},
\end{equation}
the level set of nodes in $s\in A\re(\ell+1)$ such that $s$ codes an edge with the vertex $v_{\ell}$ of
$\mathcal{H}_3$.

\begin{defn}\label{defn.mls}
A level set $X$ is {\em mutually linked} at $\ell$ if  for each $t\in X$, $t(\ell)=1$.
We say that $\ell$ is a {\em minimal level of a new mutually linked  set  in $A$} if
the set $A_{\ell,1}$ has at least two distinct members,
and for each $\ell'<\ell$,
the set $\{s\in A_{\ell,1}:  s(\ell')=1\}$ has cardinality strictly smaller than $|A_{\ell,1}|$.
In this case, we call $A_{\ell,1}$ a {\em new  mutually linked  set  in $A$}.
\end{defn}

\begin{defn}\label{defn.witness}
Given  a minimal level $\ell$  of a new mutually linked set in $A$, we say that $A_{\ell,1}$ is
{\em witnessed by the coding node $s_n$   in $A$}
if $s_i(|s_n|)=1$ for each $i\in I_{\ell}^A$, and
 either $|s_n|\le \ell$
or else
both $|s_n|> \ell$ and
 $A$ has no splitting nodes and no coding nodes of length in $[\ell,|s_n|]$.
 \end{defn}

The following is Definition 4.1 in \cite{DobrinenJML20}.

\begin{defn}[\POC]\label{defn.parallel1sProp}
Let $T\sse 2^{<\om}$ be a  strongly skew tree with coding nodes $\lgl c_n:n<N\rgl$, where $N\le\om$.
We say that $T$ satisfies the {\em \POC} if
the following hold:
Given any
 set of two or more nodes $\{t_i:i<\tilde{i}\}\sse T$
and some $\ell$ such that $t_i\re(\ell+1)$, $i<\tilde{i}$, are all distinct, and
 $t_i(\ell)=1$ for all $i<\tilde{i}$,
\begin{enumerate}
\item
There is a coding node $c_n$ in $T$ such that
 for all $i<\tilde{i}$,
$|c_n|<|t_i|$ and
$t_i(|c_n|)=1$;
we say that  $c_n$ {\em witnesses} that  $\{t_i:i<\tilde{i}\}$ is mutually linked.
\item
Letting $\ell'$ be  least such that
$t_i(\ell')=1$ for all $i<\tilde{i}$,
and  letting
$n$ be
 least  such that $c_{n}$ witnesses
 that $\{t_i:i<\tilde{i}\}$ is mutually linked,
then $T$ has no splitting nodes and no coding nodes
 of lengths strictly  between $\ell'$ and $|c_{n}|$.
\end{enumerate}
\end{defn}

Strong coding trees  in $\bS$ were defined in Subsection 4.3 of \cite{DobrinenJML20}.
That definition was streamlined in the more general paper \cite{DobrinenH_k19} for all $k$-clique-free Henson graphs.  We paraphrase here the essentials of that definition.

\begin{defn}[Strong Coding Tree]\label{def.sct}
 A {\em strong coding tree}
 is a strongly skew coding subtree $T$ of $\bS$ such that
 \begin{enumerate}
 \item
The coding nodes in $T$ are dense in $T$ and  represent a copy of $\mathcal{H}_3$ in the same order as $\bS$;
\item
$T$ satisfies the \POC.
\item
For each $n<\om$,
there is a one-to-one correspondence between the nodes in
$T\re(\ell^T_n+1)\setminus \{0^{(\ell^T_n+1)}\}$ and the  $1$-types over the graph represented by the coding nodes $\{c^T_i:i\le n\}$.
\end{enumerate}
\end{defn}

We also require that the splitting nodes in a strong coding tree $T$  between levels with coding nodes  split in reverse lexicographic order.  However,  that property is not essential to the
proofs, but rather serves to make all strong coding trees strongly similar to each other, a property which is important for topological Ramsey space theory.

The following is Definition 6.1 in \cite{DobrinenJML20}.

\begin{defn}[\SPOC]\label{defn.strPOC}
A  subtree $A$ of a strong coding tree satisfies the {\em \SPOC}
if $A$
for each  $\ell$ which is  the minimal level of a new mutually linked set  in $A$,
\begin{enumerate}
\item
The critical node in $A$ with minimal length greater than or equal to  $\ell$ is a coding node in $A$, say $c$;
\item
There are no terminal nodes in $A$ in the interval $[\ell,|c|)$  ($c$ can be terminal in $A$);
\item
$A_{\ell,1}=\{t\re (\ell+1):t\in A_{|c|,1}\}$; that is, $c$ witnesses the  mutually linked set $A_{\ell,1}$.
\end{enumerate}
\end{defn}

Note that the \SPOC\ implies the \POC.
The following is Theorem 6.3 in \cite{DobrinenJML20}; we modify the presentation, in order to avoid unnecessary definitions in this paper.

\begin{thm}[Ramsey Theorem for finite trees with \SPOC]\label{thm.MillikenIPOC}
Let $T$
 be a strong coding tree
and
let $A$ be a finite subtree  of $T$ satisfying the \SPOC.
Then for any coloring of all strongly  similar copies of $A$ in $T$ into finitely many colors,
 there is a strong coding subtree $S\le T$ such that
 each $B\sse S$ satisfying the \SPOC\ and
  strongly  similar  to  $A$  has  the same color.
\end{thm}

%%%%%%%%%%%%%%%%%%%%%%
%%%%%%%%%%%%%%%%%%%%%%
%%%%%%%%%%%%%%%%%%%%%%
%%%%%%%%%%%%%%%%%%%%%%
%%%%%%%%%%%%%%%%%%%%%%
%%%%%%%%%%%%%%%%%%%%%%

\section{Improved upper bounds}\label{sec.NewSquiggleTree}

In Theorem 8.9 of   \cite{DobrinenJML20}, we proved
  that for each   finite antichain $A$ of coding nodes,
given any coloring of  the strict similarity
 copies of $A$  inside a strong coding tree,  there is a strong coding subtree  $S$ in which   all  strictly similar copies of $A$  have the same color.
By taking  an antichain of coding nodes representing
$\mathcal{H}_3$ inside an incremental strong coding tree,
we obtained very good bounds for the  big Ramsey degrees, which we conjectured to be the exact bounds.

It turns out that those bounds are not exact for most finite triangle-free graphs (they are exact for singletons, edges, and non-edges).
However, incremental trees were  structurally on the right track.
In this section, we fine-tune that approach to obtain better upper bounds, which will be proved to be optimal in the next section.
 Subsection \ref{subsec.linkedoverview} contains the  refined version of strict similarity, which we call {\em essential pair similarity} (Definition \ref{defn.p-sim}), which will yield the exact big Ramsey degrees for triangle-free graphs.
 There, we give an overview of the main theorems of this section, Theorems \ref{thm.simple} and \ref{thm.partitionthm}, improving the upper bounds for big Ramsey degrees.
 Subsections \ref{subsec.clct} and \ref{subsec.newD}
provide the details on how to refine work in Sections 7--9  of \cite{DobrinenJML20}
   to produce those two theorems.

%%%%%%%%%%%%%%%%%%%%
%%%%%%%%%%%%%%%%%%%%

\subsection{Essential pair similarity and improved upper bounds}\label{subsec.linkedoverview}

The canonical partitions proved for the triangle-free Henson graph in this paper have a simple description, given here.
The reader convinced of the
ability of the methods in \cite{DobrinenJML20} to produce Theorems
 \ref{thm.simple}
 and
 \ref{thm.partitionthm}
below
may enjoy reading this subsection and then skipping to Section \ref{sec.brd}.
For the unconvinced reader, Subsection \ref{subsec.clct} provides the details for  how to obtain these two new theorems   from a small  but important refinement of the work in \cite{DobrinenJML20}.

   We begin with some terminology.
By an {\em antichain of coding nodes} in $\bS$,
or simply an {\em antichain},
we mean
  a set of coding nodes $A\sse \bS$ such that
no node in $A$ extends any other node in $
A$.
Since we will be working within strong coding trees,
 all our antichains  will
have meet closures which are skew.
If $A$ is an antichain,
then   the {\em tree induced by $A$}
is  the set
\begin{equation}
\{s\re |u|:s\in A\mathrm{\ and\ } u\in \cl(A)\}.
\end{equation}
Given an finite antichain $A$,
let $\ell_A$ denote the maximum length of the nodes in $A$.
For a  level set $X$, we let $\ell_X$ denote the length of the nodes in $X$.

Recall Definitions \ref{defn.linked},
\ref{defn.mls}, and \ref{defn.witness}.
 In terms of graphs,
 $s$ and $t$ are base-linked if for each coding node $c_m\contains s$ and for each coding node $c_n\contains t$,
the least $i$ such that $v_m\, E\, v_i$ equals the least $i$ such that $v_n\, E\, v_i$.
This means that  $s$ and $t$ code  first edges with a common vertex in $\mathcal{H}_3$.
 Given a subset $A\sse\bS$, we say that $A$ has a {\em new linked pair} at level $\ell$ if and only if there is a pair of nodes $\{s,t\}\sse A\re (\ell+1)$
 such that $s(\ell)=t(\ell)=1$,
$s(\ell')$ and $t(\ell')$ are never  both $1$ for any
 $\ell'<\ell$,
 and for all $u\in (A\re (\ell+1))\setminus \{s,t\}$,
 $u(\ell)=0$.
 We call $\{s,t\}$ an  {\em essential  pair} for level $\ell$.
 Notice
that
for  a base-linked pair, the minimal $\ell$ for which $s(\ell)=t(\ell)=1$ satisfies $\ell\le |s\wedge t|$; so
  by definition, an essential pair  is not base-linked.
 Thus, linked pairs are either base-linked or essential, and never both.

\begin{defn}[Essential Pair Similarity]\label{defn.p-sim}
Suppose $S$ and $T$ are meet-closed sets, and
let $\lgl k_i:i<M\rgl$ and $\lgl \ell_i:i<N\rgl$ enumerate the levels of new  essential linked pairs in $S$ and $T$, respectively.
(This excludes levels of new mutually linked sets of size greater than two.)
We say that a map $f:S\ra T$
 is an {\em essential pair similarity map}  (or {\em ep-similarity}) if and only if
 $f$  is a strong similarity map and
  additionally the following   hold:
$M=N$,
and for each $i<M$,
letting $d$ be the  critical node in $S$  with  least length greater than $k_i$,
if
 $\{s_0,s_1\}\sse S\re |d|$
 is the essential pair for level $k_i$,
 then
$\{f(s_0),f(s_1)\}$ is the essential pair in $T\re|f(d)|$ for  level $\ell_i$.

Given finite antichains of coding nodes $A,B$ in a strong coding tree,
we say that $A$ and $B$ are {\em essential pair similar} (or {\em ep-similar})
if and only if
$A$ and $B$ are strongly similar,
and the strong similarity map from $\cl(A)$ to $\cl(B)$ is an  essential pair similarity.
When $A$ and $B$ are ep-similar, we write $A\psim B$.
\end{defn}

\begin{rem}
Strict similarity  was the  structural property characterizing  our upper bounds for big Ramsey degrees  in  \cite{DobrinenJML20}.
As we are not using strict similarity directly in this paper, we refer the reader to
Definition 8.3 in \cite{DobrinenJML20}.
We point out that  strict similarity implies ep-similarity, and not vice versa.
That is, ep-similarity is a courser equivalence relation than strict similarity.
The improvement of our upper bounds in this paper is due to constructing a coding tree $S$ which
represents a copy of the triangle-free Henson graph and in which every two ep-similar antichains are actually strictly similar.
This will clear away all superfluous strict similarity types, leaving us with an exact characterization of the big Ramsey degrees.
\end{rem}

 By applying Theorem \ref{thm.MillikenIPOC} finitely many times, using the methods  in \cite{DobrinenJML20} from Sections 7--8
 but substituting
  a {\em canonically linked}  subtree $S$
  (see Definition \ref{defn.canonlinked})
   in place of the incremental strong coding subtree in Lemma 7.5,
  we arrive at the following improvement of Theorem 8.9 from \cite{DobrinenJML20}.

\begin{thm}[Ramsey Theorem for Essential Pair Similar Antichains]\label{thm.simple}
Given a  strong coding tree $T$ and  a finite triangle-free graph  $\bG$, suppose
$h$ colors all  antichains $A$ of coding nodes in $T$
representing a copy of $\bG$ into
 finitely colors.
Then there is a  canonically linked coding tree $S\sse T$
such that all
ep-similar  antichains
of coding nodes in $S$
  have the same $h$-color.
\end{thm}

\begin{defn}[$\Simp(\bG)$]\label{defn.SimpG}
Given  a finite triangle-free graph $\bG$,
let
$\Simp(\bG)$ denote
a set of representatives from
the  different ep-similar equivalence classes of  strongly skew  antichains $A$
 representing  $\bG$
with the property that
 any coding node $c_n^A$ in $A$ with passing number $0$ at another coding node $c_m^A$ (where $m<n$)
 is linked with $c_m^A$.
 \end{defn}

The work of
 Section 9 of \cite{DobrinenJML20}
 with a small but important modification made in Subsection \ref{subsec.newD}
  yields the following  theorem.

\begin{thm}[Improved Upper Bounds]\label{thm.partitionthm}
Let  $\bG$ be a finite triangle-free graph, and
 let $f$ color all the copies of $\bG$ in $\mathcal{H}_3$ into finitely many colors.
Then there is a subgraph $\mathcal{H}'$ of $\mathcal{H}_3$, which is isomorphic to $\mathcal{H}_3$, such that
$f$ takes no more than $|\Simp(\bG)|$-many colors in $\mathcal{H}'$.
Hence,
\begin{equation}
T(\bG,\mathcal{H}_3)\le |\Simp(\bG)|.
\end{equation}
\end{thm}

We will prove that
 $T(\bG,\mathcal{H}_3)= |\Simp(\bG)|$
 in Section 4.

%%%%%%%%%%%%%%%%%%%%
%%%%%%%%%%%%%%%%%%%%

\subsection{Canonically linked coding trees}\label{subsec.clct}

In this subsection, we
improve the main result of Section 7 in \cite{DobrinenJML20}.
We will show  in Lemma \ref{lem.squiggliesttree}
that in any strong coding tree $T$, there is a  {\em canonically linked}
(see Definition \ref{defn.canonlinked}) coding
subtree $S$  and a subset $W$ of witnessing coding nodes with the following property:
Given any antichain  $A$ of coding nodes in $S$, there is a set of coding nodes $W_A$ in $W$ so that
$A\cup W_A$ has the \SPOC.
The canonical linked-ness of $S$
serves to  get rid of  most the superfluous strict similarity types  which remained in  our upper bounds in \cite{DobrinenJML20}:
All   ep-similar antichains in a
canonically linked coding tree  are strictly similar.
The remaining few  superfluous strict similarity types will be  eradicated by our construction of the antichain
$\bD$ in Lemma \ref{lem.newbD}.

For two subsets $X,Y$ of some level set $Z$,
we  say that $X$ is {\em lexicographically less than} $Y$, and write $X<_{\mathrm{lex}} Y$,
 if and only if,
letting $\lgl x_i:i<m\rgl$ and $\lgl y_i:i<n\rgl$ be the lexicographically increasing enumerations of $X$ and $Y$, respectively,
then  either
(a)
$x_i<_{\mathrm{lex}} y_i$
for the $i$ least such that $x_i\ne y_i$,
or (b)
$\lgl x_i:i<m\rgl$ is an initial segment of  $\lgl y_i:i<n\rgl$.
The following is Definition 7.1 in \cite{DobrinenJML20}, where it was called {\em Incremental  Parallel $1$'s}.

\begin{defn}[Incremental  Linked Sets]\label{defn.incrementalpo}
Let $Z$ be a  finite
 subtree of a strong coding tree $T$,
 and let $\lgl \ell_j:j<\tilde{j}\rgl$
list in increasing order the minimal lengths
of new mutually linked sets in $Z$.
We say that
$Z$ has {\em incremental   linked sets}
 if
the following holds.
For each $j<\tilde{j}$ for which
\begin{equation}
Z_{\ell_j,1}:=\{z\re (\ell_{j+1}):z\in Z, \ |z|>\ell_j,\mathrm{\  and \ } z(\ell_j)=1\}
\end{equation}
has size at least three,
letting $m$  denote the length of the longest critical node in $Z$ below $\ell_j$,
for each
proper subset $Y\subsetneq Z_{\ell_j,1}$  of cardinality at least two,
there is a $j'<j$ such that  $\ell_{j'}>m$,
$Y_{\ell_{j'},1}:=\{y\re(\ell_{j'}+1):y\in Y$ and $y(\ell_{j'})=1\}$ has the same size as $Y$,
and
$Y_{\ell_{j'},1}=Z_{\ell_{j'},1}$.

We shall say that an infinite  tree $S$ has {\em incremental linked sets} if for  each
$\ell<\om$, the
initial subtree $S\re \ell$ of $S$
 has incremental  linked sets.
\end{defn}

 We shall use canonical completions to construct an incremental coding tree with a minimal number of ep-similarity types.

\begin{defn}[Canonical Completion of  a Linked Pair]\label{defn.canoncomplete}
Suppose $A$ is a subtree of a strong coding tree $T$ and suppose
$X=A\re(\ell+1)$   has a  linked pair at $\ell$.
We call a level set  $Y$ end-extending $X$  in $A$
 a  {\em canonical completion} of $X$ if and only if
 there is no splitting or coding node in $A$ in the interval $[\ell,\ell_Y]$ and the following hold:

Let $\lgl M_i:i<\tilde{i}\rgl$  enumerate
 all maximal pairwise linked sets in $X$.
 List  all subsets  of size three from all  $M_i$, $i<\tilde{i}$,   in
 lexicographic order as $\lgl P_{3,j}:j<k_3\rgl$.
 Then  list in lexicographic order all subsets of size four from all
 $M_i$, $i<\tilde{i}$, as $\lgl P_{4,j}:j<k_4\rgl$,
 etc., until all the sets $M_i$
 appear as  $P_{|M_i|, j}$ for some $j<k_{|M_i|}$.
 Let  $\tilde{m}=\max(|M_i|:i<\tilde{i})$.
 Then for each $3\le i\le\tilde{m}$ and $j<k_i$,
 there is a level $\ell' \in (\ell,\ell_Y)$ such that
 $Y_{\ell',1}$ is  mutually linked set end-extending $P_{i,j}$.
 Moreover, letting $ \ell_{i,j}$ be the least level above $\ell'$ where
  $Y_{\ell',1}$ is a mutual linked set end-extending $P_{i,j}$,
 the sequence $\lgl \ell_{3,j}:j<k_3\rgl{}^{\frown}\dots^{\frown}
 \lgl \ell_{\tilde{m},j}:j<k_{\tilde{m}}\rgl$ is an increasing sequence.
\end{defn}

%Go back and clean this definition.

Thus, a canonical completion incrementally adds linked sets of the same size   in lexicographic order,
and then repeats this process  for sets of the next largest size  until it completes this process up to a new  mutually linked set
  for each new maximal pairwise linked set.
One can think of this as supersaturating the tree with all new  linked sets  in a canonical manner which will not negatively affect branching capabilities.
By this, we mean that whenever $X$ is a pairwise linked set, given any $s$ and $t$ in $X$ and any coding nodes $c_m,c_n$  ($m<n$) extending $s,t$, respectively,
then $c_n$ must have passing number $0$ at $c_n$.
Thus, adding new mutually linked sets among a pairwise linked set does not affect the ability of the tree to code a copy of $\mathcal{H}_3$.
Note that if a linked pair is not included in any larger pairwise linked set,
then that pair is its own canonical completion;
 no other linked pairs need be added.

The  following     refines the notion of incremental linked sets.
This is the fundamental notion behind the canonical partitions, which provide  exact big Ramsey degrees.

\begin{defn}[Canonically Linked]\label{defn.canonlinked}
Let $A$ be a
 subtree of a strong coding tree $T$,
 and let $\lgl c^A_n:n<N\rgl$ enumerate the coding nodes in $A$.
We say that
$A$ is {\em canonically linked}
 if for each $n<N$,
the following holds.
Let $\ell_*=0$ if $n=0$; otherwise, let $\ell_*=\ell^A_{n-1}+1$.

\begin{enumerate}

\item
For each splitting node $s$ in $A$  in the interval
$[\ell_*,\ell^A_n)$,
the  minimal new mutually linked set
 in  $A$ above $|s|$
is a
pair of nodes $s_0,s_1$
extending $s^{\frown}0,s^{\frown}1$, respectively, such that
$s_0(\ell)=s_1(\ell)=1$ for some $\ell$.
Moreover, above this $\ell$, there is a canonical completion
 before any new splitting node or any other new linked pairs occur.

\item
Once we have performed the canonical completion on the maximal splitting node in $A$ below $c_n^A$,
let $\lgl Q_q:q<\tilde{q}\rgl$
enumerate in lexicographic order all
 pairs of nodes in $A_{\ell^A_n,1}$.
 Add a  linked pair for $Q_0$ and then perform the canonical completion.
Then add a  linked pair for $Q_1$ and then perform the canonical completion.
And so on until $Q_{\tilde{q}-1}$ has been taken care of.
After this, extend to the level of $c^A_n$.
\end{enumerate}
\end{defn}

In particular, whenever a new linked  pair   occurs,
then the canonical completion occurs before any other critical node or other new  linked pair   occurs.

 \begin{observation}\label{obs.clpersists}
Any subtree of a canonically linked coding tree is again canonically linked.
Moreover, for any two antichains $A,B$ of coding nodes in  a canonically linked coding tree,
$A$ and $B$ are strictly similar if and only if $A\psim B$.
\end{observation}

\begin{rem}
 A  canonically linked  tree cannot be  a strong coding tree in the sense of the definition given in \cite{DobrinenJML20}.
 This is because we stipulated that strong coding trees have the property that taking leftmost extensions never adds a new linked set, whereas in canonically linked trees, any new maximal pairwise linked set  $X$ will be
 followed by a new mutually linked set  end-extending $X$.
 However, this does not affect the availability of passing numbers needed to construct subcopies of $\mathcal{H}_3$, for if $X$ is pairwise linked, then
 whenever one node in $X$ is extended to a coding node $c$, any other node extending a node in $X$ must have passing number $0$ at $c$.
Thus,  adding the  mutually linked set end-extending $X$ does not affect the ability to extend to a subtree coding a copy of $\mathcal{H}_3$.

In retrospect, we could have used {\em maximally linked} trees  from the outset in \cite{DobrinenJML20},
where instead of adding the canonical completion to each new linked pair, we simply
add a mutually linked set   immediately  above each new maximal pairwise linked set.
All proofs in that paper could be modified to hold
 using such coding trees.
However, the way we defined  strong coding trees,
it is possible that within a given strong coding tree $T$,
new linked pairs might occur before we could add a mutually linked set.
 Thus,  for the sake of logic,
  we shall use what has been proved there, rather than
 rehashing all those proofs or
asking the reader to  believe  without proof  that
the work in  \cite{DobrinenJML20} holds if we replace strong coding trees with  canonically linked or maximally linked coding trees.
\end{rem}

 By a {\em canonically linked coding tree}, we mean a strongly skew canonically linked coding  subtree $S$ of $\bS$
 satisfying (1) and (3) of Definition \ref{def.sct}.
Every strong coding tree contains a canonically linked  coding subtree, as we shall show below.
The following is a revised version of   Definition 7.4 in \cite{DobrinenJML20}.
Given a node $w\in 2^{<\om}$, we let $w^{\wedge}$ denote the maximal initial segment of $w$ which is a sequence of $0$'s.
We say that  $W\sse T$ is a {\em set of witnessing coding  nodes} for a
canonically linked coding tree  $S\sse T$
if and only if
each new  mutually linked set  $X\sse S$
 is witnessed by a coding node $w\in W$ such that
$|w^{\wedge}|$ is less than the level of $X$ and $w$ is linked with no member of $S$.

The next  Lemma  says that  inside any strong coding tree $T$,  we can construct a canonically linked coding subtree and a set of canonical witnessing coding nodes.
 This  improves
 Lemma 7.5 in \cite{DobrinenJML20}.

\begin{lem}[Canonically linked coding tree]\label{lem.squiggliesttree}
Let $T$ be a strong coding tree.
Then there is a canonically linked  coding tree $S\sse T$ and a set of  witnessing coding nodes $W\sse T$ such that  each
 new mutually  linked set in $S$ is
 witnessed in $T$ by  a coding node in $W$.
\end{lem}

\begin{proof}
Let  $\lgl d^T_m:m<\om\rgl$ denote the critical nodes in $T$ in order of  increasing length.
Let
$\lgl m_n:n<\om\rgl$ denote the indices
such that $d^T_{m_n}=c^T_n$, so that  the $m_n$-th critical node in $T$ is the $n$-th coding node in $T$,
and let $T(m)$ denote  the level set $T\re |d^T_m|$.
We shall construct a canonically linked subtree $S$ of $T$ so that $S$ is strongly similar to $T$.
We will let $d^S_m$ denote the $m$-th critical node of $S$, and  $S(m)=S\re |d^S_m|$.

Since $T$  is a strongly skew tree coding
$\mathcal{H}_3$,
$d^T_0$  and $d^T_1$ are  splitting nodes of $T$ which are members of $0^{<\om}$ and  $d^T_2=c^T_0$ is a coding node (hence not in $0^{<\om}$).
Moreover, $\bigcup_{m<3}T(m)$ has no essential linked pairs, so it is canonically linked.
Thus, we let $d^S_m=d^T_m$,  $S(m)=T(m)$,
and $W_m=\emptyset$,
for all $m<3$.

Given   $m\ge 3$, suppose  we have chosen
 $S(k)$ for all $k<m$ so that $\bigcup_{k<m}S(k)$ is
  canonically linked and strongly similar to
  $\bigcup_{k<m}T(k)$
  Moreover,
  suppose we have chosen a  level set  extension $S(m-1)^+\sse T$  of $S(m-1)$
  which is a canonical completion of  $S(m-1)$
   and a set of  coding nodes $W_m\sse T$
   witnessing each new linked set in the canonical completion.
   We also suppose that
  $S(m-1)^+$ has no  predetermined new linked sets in $T$.
  (This was called {\em no predetermined new parallel $1$'s} in \cite{DobrinenJML20}.
  It means that it is possible to extend
  the level set $S(m-1)^+$ in $T$ without adding any new linked sets.)
  Let $T^+(m-1)$   denote the set of immediate successors  $T(m-1)$
   in $\widehat{T}$, and
  let $f:T(m-1)^+\ra S(m-1)^+$ be the lexicographic-preserving bijection between these two level sets.
  Let $t_*$ be the node in $T(m-1)^+$ such that $t_*\sse d^T_m$,
and let   $s_*=f(t_*)$.

Suppose first that $d^T_m$ is a splitting node.
Take  $d^S_m$  to  be
 any splitting node in $T$ extending $s_*$, and extend all other nodes in $S(m-1)^+$  along the leftmost paths in $T$ to the same length as $d^S_m$.
 These nodes comprise the level set $S(m)$.
 Now take the set $\BL_S(d^S_m)$ of all nodes in
 $S(m)$ which are base-linked with $d^S_m$.
 Let $Y$ be  a level set  in $T$ end-extending $\BL_S(d^S_m)$  such that $Y$ is
  a canonical completion of  $\BL_S(d^S_m)$.
As the canonical completion $Y$ is being constructed,  take $W_m$ to be a set of witnessing coding nodes in $T$, similar to the construction in Lemma   7.5 in \cite{DobrinenJML20}.
 Finally, extend all nodes in $S(m)\setminus\BL_S(d^S_m)$
 leftmost in $T$ to nodes of the length in $Y$,
 and let $S(m)^+$ be the union of $Y$ along with these nodes.
 Then $S(m)^+$ has no predetermined new linked sets, and all of its linked sets occur either in $S(m)$ or $Y$.

Now suppose that $d^T_m$ is a coding node.
Let  $\ell$ denote $|d^T_m|$, and recall that
 $T_{\ell,1}$ denotes the collection of nodes in
 $T\re (\ell+1)$ which have passing number $1$ at $d^T_m$.
Let $Y$ denote  the set of nodes in $T(m-1)^+$ which extend to a node in $T_{\ell,1}$,
and let $Z=f[Y]$.
Let
$\lgl P_{j}:j<\tilde{j}\rgl$ enumerate in lexicographic order the collection of pairs of members in $Z$.

Extend the nodes in $P_0$ to a linked pair in $T$ of the same length, say $P_0'$.
Let $X_0$ be the union of $P_0'$  along with leftmost extensions in $T$ of all other nodes in $S(m-1)^+$
to the same length.
Look at all  the maximal pairwise linked sets in $X_0$, and
 take $Y_0$ to be a level set end-extending $X_0$ in $T$ such that
$Y_0$ is a canonical completion of $X_0$.
As this canonical completion is being formed,
add new witnessing coding nodes into the set $W_m$
 (similar to the construction in Lemma   7.5 in \cite{DobrinenJML20}).
In general, for  $j<\tilde{j}-1$, given $Y_j$,
let $P'_{j+1}$ be the set of nodes in $Y_j$ extending the nodes in $P_{j+1}$.
Let $X_{j+1}$ be an end-extension of $Y_j$ in $T$
which adds a linked pair above
$P_{j+1}'$.
Then perform the
 canonical completion of $X_{j+1}$
to obtain a level set extension $Y_{j+1}$  in $T$
while adding coding nodes to the set $W_m$ to witness each new linked set.
At the end of this process, we have a level set $Y_{\tilde{j}-1}$.
By Lemma
4.18 in \cite{DobrinenJML20},
we can extend $Y_{\tilde{j}-1}$ to a level set $S(m)$ in $T$ so that the coding node in $S(m)$ extends $s_*$,
and the lexicographic-preserving map from $T(m)$ to $S(m)$ preserves the passing types at the coding node in this level.
It follows that $\bigcup_{k\le m}T(k)$ is strongly similar to $\bigcup_{k\le m}S(k)$.

To finish, let $S=\bigcup_{m<\om} S(m)$ and let $W=\bigcup_{m<\om}W_m$.
Then $S$ is strongly similar to $T$ (and hence, represents a copy of $\mathcal{H}_3$), $S$ is canonically linked, and
all new linked sets  in $S$ are  witnessed
by coding nodes in $W$.
\end{proof}

It follows from the construction in the previous lemma that  for each antichain $A\sse S$,
there is a set of coding nodes $W_A\sse W$ such that
$A\cup W_A$ satisfies the \SPOC.
In particular, we can choose $W_A$ so that   $A\cup W_A$   will be an
{\em envelope} of $A$ (in the terminology of Section 8 of \cite{DobrinenJML20}).

Observe that   for a  canonically linked coding tree $S$, whenever  $Y,Y'$ are level sets in $S$ with
$Y'$ end-extending $Y$,
 then
 $Y'$ has no new mutually linked sets over $Y$ if and only if
 $Y'$ has no new  linked pairs.
Thus,  inside $S$, the notion of strict similarity
(Definition 8.3 in \cite{DobrinenJML20})
reduces  simply to   essential pair   similarity.
By replacing
the uses of Lemma 7.5 of \cite{DobrinenJML20} with Lemma \ref{lem.squiggliesttree} in the proof
of Theorem 8.9 in \cite{DobrinenJML20},
we obtain  Theorem \ref{thm.simple}.

%The referee might want us to add the proof.

%%%%%%%%%%%%%%%%%%%%%%
%%%%%%%%%%%%%%%%%%%%%%
%%%%%%%%%%%%%%%%%%%%%%

\subsection{Improved antichain of coding nodes $\bD$ representing $\mathcal{H}_3$}\label{subsec.newD}

In  Lemma 9.1 in \cite{DobrinenJML20}, we showed that within any strong coding tree, there is an antichain $\bD$ of coding nodes which represent a copy of the triangle-free Henson graph.
For the proof of Theorem  \ref{thm.partitionthm}, we will need to make a slight modification to this  construction
in order to sweep away the  remaining superfluous ep-similarity types.
We do this by linking any two coding nodes in the antichain where the longer one has passing number $0$ at the shorter coding node.

\begin{lem}\label{lem.newbD}
Let $S$ be a  canonically incremental coding tree.
Then there is an infinite  antichain of coding nodes  $\bD\sse S$  which code
 $\mathcal{H}_3$ in exactly the same way that $S$ does with the following additional property:
Whenever $m<n$ and $c^{\bD}_n$ has passing type $0$ at $c^{\bD}_m$,
 then $c^{\bD}_m$ and $c^{\bD}_n$ are linked.
\end{lem}

\begin{proof}
We will  construct
an antichain of coding nodes
 $\bD\sse S$ which codes a copy of $\mathcal{H}_3$ in the same order as $S$.
 It is important to notice that, while taking leftmost extensions in $S$ can add new mutually linked sets (indeed this is the point of being canonically linked),
 it will never add a new essential linked pair.
Thus,  leftmost extensions of any unlinked pair in $S$ yield another unlinked pair in $S$.

We use the notation
$\lgl c^S_n:n<\om\rgl$ to denote the $n$-th coding node of $S$,  and $\ell^S_n$ to denote $|c^S_n|$.
We let
$\lgl d^S_m:m<\om\rgl$  denote the critical nodes (coding and splitting) of $S$, and $m_n$ denote the index such that $d^S_{m_n}$ equals the coding node $d^S_n$.
Likewise, we will use
$c^{\bD}_n$ to denote the $n$-th coding node in $\bD$, and $\ell^{\bD}_n$ to denote its length.
The set of  nodes in  $\bD\setminus\{c^{\bD}_n\}$ of length  $\ell^{\bD}_n$
 shall be indexed  as $\{d_s:s\in S\re l^S_n\}$.
We will construct $\bD$ so that  for each $n$,
the node of length $l^{\bD}_{n}+1$ which  is going to be extended to the next  coding node $c^{\bD}_{n+1}$ will  split at a  level lower than
 any of the other nodes of length $l^{\bD}_{n+1}$  split in $\bD$.

%To aid in the construction, for each $m<\om$,
%we will let $X_m$ denote the level set $S\re(|d^S_{m+1}-1)$ and build a level set $Y_m$ end-extending $\bD\re |d^{\bD}_m+1| so that the lexicographic-preserving map $\psi_m:X_m\ra Y_m$ ???

Define  $d_{0}^{\bD}=  d^S_0$, the root of $S$,
 and
let $\bD(0)=\{d_{0}^{\bD}\}$, the $0$-th level of $\bD$.
As the node $d^S_0$ splits in $S$,
so also the node $d^{\bD}_0$ will split in $\bD$.
Let $Y_0$ denote the set of the two immediate successors
$d_{0}^{\bD}$ in $\widehat{S}$.

For the induction step,
suppose $m\ge 1$ and we have constructed
$\bigcup_{k<m}\bD(k)\sse S$ so that it is  ep-similar to $\bigcup_{k<m}S(k)$.
Let $n$ be the index of the longest coding node in $S$ in $\bigcup_{k<m}S(k)$.
We have three cases:
\vskip.1in

\noindent{\bf Case I.} $d^S_m$ is a splitting node and $d^S_{m-1}$ is the coding node $c^S_n$.
\vskip.1in

Define  $X$  to be the set of immediate successors in $\widehat{S}$ of the level set $S(m-1)$, and define $Y$   to be the set of immediate successors in $\widehat{S}$ of the level set
$D(m-1)\setminus\{c^{\bD}_n\}$, respectively.
Let $\psi:X\ra Y$ be the lexicographic preserving bijection.
Define  $s_*$ to be the node in $X$ which extends to the splitting node $d^S_m$, and let $t_*=\psi(s_*)$.
Let $x_*$ be the node in $X$ which extends to the next coding node in $S$, and let $y_*$ denote $\psi(x_*)$.
Note that $s_*$ and $x_*$ are distinct, because every splitting node in $S$ has an extension with passing number $1$ at the next coding node, while every
 node which is
base-linked with $x_*$ must have passing number $0$ at the next coding node, so does not  extend to a splitting node in this interval.

First extend $y_*$ to a splitting node $y'_*$ in $S$. Then let $Y'$ be the level set of nodes of length
$|y'_*|+1$
consisting of  ${y'_*}^{\frown}0$ and  ${y'_*}^{\frown}1$ as well as leftmost extensions in $S$ of the nodes in $Y\setminus\{y_*\}$
to the length  $|y'_*|+1$.
After this, extend the node in $Y'$ extending $t_*$ to a splitting node in $S$, and label this splitting node $d^{\bD}_m$.
Then let
$\bD(m)$ consist of the node  $d^{\bD}_m$ along with leftmost extensions of the nodes in
$Y'\setminus \{d^{\bD}_m\}$ in $S$.
Note that $\bD(m)$ has one more node than $S(m)$, precisely the node extending ${y'_*}^{\frown}1$; label this node $e_{n+1}$.
This is the node that will be extended to the coding node $c^{\bD}_{n+1}$.
This  construction adds no new linked pairs over $\bD\re (\ell_n^{bD}+1)$.
\vskip.1in

\noindent{\bf Case II.} $d^S_m$ and $d^S_{m-1}$  are both splitting nodes.
\vskip.1in

Let $e$ denote the node in $D(m-1)$ extending $e_{n+1}$.
Define  $X$  to be the set of immediate successors in $\widehat{S}$ of the level set $S(m-1)$, and define $Y$   to be the set of immediate successors in $\widehat{S}$ of the level set
$D(m-1)\setminus\{e\}$, respectively.
Let $\psi:X\ra Y$ be the lexicographic preserving bijection.
As in Case I, let
  $s_*$ be the node in $X$ which extends to the splitting node $d^S_m$, and let $t_*=\psi(s_*)$.
  Then extend $t_*$ to a splitting node in $S$ and label it $d^{\bD}_m$.
  Let $\bD(m)$ be the collection of the
   leftmost extensions in $S$ of the nodes in $Y\setminus \{e\}$ along with $d^{\bD}_m$.
\vskip.1in

\noindent{\bf Case III.} $d^S_m$ is a coding node.
\vskip.1in

In this case, $d^S_m=c^S_{n+1}$, and $d^S_{m-1}$ is a splitting node.
Let $X$ denote the set of immediate successors of $S(m-1)$ in $\widehat{S}$.
Let $e$ denote the node in $\bD(m-1)$ extending $e_{n+1}$,
and let $Y$ denote the set of immediate successors of
$\bD(m-1)\setminus\{e\}$ in $\widehat{S}$.
Let $\psi$ be the  lexicographic preserving bijection from $X$ to $Y$.

As a preparatory step, let $c$ be a coding node in $S$ extending $e$ of long enough length   that
there is a level set extension $Y'$ of $Y$ in
$\widehat{S}$ of length $|c|$ such that
the lexicographic preserving map from $S(m)$ to $Y'$ preserves passing numbers  at the coding node at these levels.
Since $S$ is canonically linked, this automatically is inherited by  $Y'$; that is, $Y'$ is canonically linked.

Now, we  will extend $c$ along with  the nodes in $Y'$ to construct $\bD(m)$
 so that for each non-coding node $t\in\bD(m)$ with passing number $0$ at $c^{\bD}_m$,
 $t$ is linked with $c^{\bD}_m$.
Let $Y'_0$ denote those nodes in $Y'$ which have passing number $0$ at $c$, and let $\lgl y_j:j<\tilde{j}\}$ be the enumeration of $Y'_0$ in
 lexicographic order.
Take $z_0$ extending $y_0$ and $u_0$ extending $c$ in $S$ so that $z_0$ and $u_0$ are linked.
Given $z_j$ and $u_j$, where $j<\tilde{j}-1$,
take $z_{j+1}$ extending $y_{j+1}$ and $u_{j+1}$ extending $u_j$ in $S$ so that $z_{j+1}$ and $u_{j+1}$ are linked.
After this process is complete,
let $Y''$ be the  level set of the leftmost extensions
 of the nodes $\{z_j:j<\tilde{j}\}$ to the length of $z_{\tilde{j}-1}$.
Then extend $Y''\cup\{u_{\tilde{j}-1}\}$ to a level set $\bD(m)$ so that the node in $\bD(m)$ extending  $u_{\tilde{j}-1}$ is a coding node,  label it $c^{\bD}_{n+1}$,
and the
lexicographic preserving map from $S(m)$ to $\bD(m)\setminus\{c^{\bD}_{n+1}\}$ preserves passing numbers.
\vskip.1in

This concludes the construction of $\bD$ satisfying the Lemma.
\end{proof}

Recall Definition \ref{defn.SimpG}, where $\Simp(\bG)$ was defined.
Notice that
for any antichain $A\sse\bD$,
$A$ is ep-similar to a representative in
 $\Simp(\bG)$.
Replacing the uses of Theorem 8.9  and Lemma 9.1 of \cite{DobrinenJML20}
 with  applications of Theorem \ref{thm.simple} and Lemma \ref{lem.newbD} in
 the proof of
Theorem 9.2 in \cite{DobrinenJML20}
yields  our Theorem \ref{thm.partitionthm}.

In the next section, we will prove that each of the ep-similarity types in
 $\Simp(\bG)$ persist in any subcopy of $\mathcal{H}_3$ contained in the one coded by $\bD$.
It will follow that  the big Ramsey degree
$T(\bG,\mathcal{G}_3)$ is exactly the cardinality of
  $\Simp(\bG)$.

%%%%%%%%%%%%%%%%%%%%%%
%%%%%%%%%%%%%%%%%%%%%%
%%%%%%%%%%%%%%%%%%%%%%
%%%%%%%%%%%%%%%%%%%%%%
%%%%%%%%%%%%%%%%%%%%%%
%%%%%%%%%%%%%%%%%%%%%%

\section{Canonical Partitions}\label{sec.brd}

In this section, we  prove that  for a given finite triangle-free graph $\bG$,
each of the types in $\Simp(\bG)$
 persists in every subcopy of $\mathcal{H}_3$.
 This produces canonical partitions of the copies of $\bG$ in $\mathcal{H}_3$,  characterizing the exact big Ramsey degree of $\bG$ in $\mathcal{H}_3$.

Fix a canonically linked  coding tree $S$ (recall Lemma \ref{lem.squiggliesttree}) and an
 antichain  of coding nodes $\bD\sse S$  such that $\bD$ represents a copy of $\mathcal{H}_3$.
(The construction of such a $\bD$ is done in Lemma 9.1 of \cite{DobrinenJML20}.)

\begin{thm}[Persistence]\label{thm.persistence}
Let  $D$ be any subset of $\bD$    representing  a copy of $\mathcal{H}_3$.
Given any antichain  of coding nodes $A\sse S$,
there is an essential pair  similarity embedding of $A$ into $D$.
It follows that every essential pair  similarity type of an antichain in $S$ persists in $D$.
\end{thm}

\begin{proof}
Let $\bC$ denote $\{c_n:n<\om\}$, the set of all coding nodes in $\bS$.
Fix an antichain of coding nodes $D\sse\bD$ representing   $\mathcal{H}_3$.
Without loss of generality, we may assume that $D$ represents $\mathcal{H}_3$ in the same order that $\bS$ does.
Let $\{c^D_n:n<\om\}$ enumerate the coding nodes in $D$, in order of increasing length.
Then the map
$\varphi:\bC\ra D$ via
$\varphi(c_n)=c^{D}_n$
is {\em passing number preserving},
meaning that
whenever $m<n$,  then
\begin{equation}
\varphi(c_n)(|\varphi(c_m)|)=c_n(|c_m|).
\end{equation}

 Define
 \begin{equation}
 \overline{D}
 =\{c^{D}_n\re |c^{D}_m|: m\le n<\om\}.
 \end{equation}
 Then $\overline{D}$ is a union of level sets, but not meet-closed.
 We extend the map $\varphi$ to a map $\bar{\varphi}:\bS\ra\overline{D}$ as follows:
 Given $s\in \bS$, let $n$ be least such that $c_n\contains s$, and let  $m$ be the integer such that $|s|=|c_m|$,
  and define $\bar{\varphi}(s)=\varphi(c_n)\re|\varphi(c_m)|$;
  that is, $\bar{\varphi}(s)=c^{D}_n\re|c^{D}_m|$.
Notice that
$\bar{\varphi}$ is  again passing number preserving:
Given  $s,t\in \bS$ with $|s|=|c_m|<|t|$,
and  given $n$ least such that $c_n\contains t$,
we have
\begin{equation}
\bar{\varphi}(t)(|\bar{\varphi}(s)|)
=\varphi(c_n)(|\varphi(c_m)|)
=c^D_n(|c^D_m|)
=c_n(|c_m|)
=t(|s|).
\end{equation}

\begin{observation}\label{obs.Fact1}
If $m<n$ and $c_n(|c_m|)=1$,
then $c^D_n(|c^D_m|)=1$; hence $c_n^D$ and $c_m^D$ have no parallel $1$'s.
\end{observation}

In what  follows, for $s\in\bS$,  we let
$\widehat{s}$ denote the cone of all $s'\in \bS$ extending $s$.
A subset $X\sse\widehat{s}$  is {\em cofinal} in $\widehat{s}$ if for each $s'\contains s$ in $\bS$, there is an $x\in X$ such that $x\contains s'$.
For $t\in\overline{D}$, $\widehat{t}$ denotes the set of all $t'\in\overline{D}$ extending $t$.
Since $\varphi$ is a map from $\bC$ onto $D$, it follows that
 for a subset $L\sse\overline{D}$,
$\varphi^{-1}[L]$ is the set of coding nodes $c\in\bC$ such that $\varphi(c)\in L$.
We work with subsets $L$ of $\overline{D}$ rather than just of $D$ because we shall be interested in cones above members of $\overline{D}$, and allowing this flexible notation will reduce the need for
extra symbols throughout.

\begin{defn}\label{def.large}
Let $L$ be a subset of $\overline{D}$.
Given  $s\in\bS$, we say that $L$ is {\em $s$-large} if and only if $\varphi^{-1}[L]\cap\widehat{s}$ is cofinal in
$\widehat{s}$.
We say that $L$ is
{\em large} if and only if there is some
$s\in\bS$ for which $L$ is $s$-large.

We say that $L$ is {\em $0$-large} if and only if $L$ is $s$-large for  some $s\in 0^{<\om}$.
Call $s'$ a {\em $0$-extension of $s$} if and only if  $\contains s$ and  for each $|s|\le i<|s'|$, $s'(i)=0$.
We say that   $L$ is  {\em $s$-$0$-large}
 if and only if $L$ is $s'$-large for some $0$-extension $s'$ of $s$.
\end{defn}

The  next  observation
follows immediately from the definitions.

\begin{observation}\label{obs.Fact4}
If $L\sse\overline{D}$ is $s$-large, then $L$ is $s'$-large for every $s'$  extending $s$.
In particular,  $L$ is $s'$-large  for each $s'$ which $0$-extends $s$, which implies that
 $L$ is $s$-$0$-large.
\end{observation}

The following series of lemmas will aid in building
an ep-similarity copy of a given antichain from $S$
 inside of $D$.

\begin{lem}\label{lem.Fact2}
Suppose $t$ is in $\overline{D}$  and
  $\widehat{t}$ is $0$-large.
  Then
$\varphi^{-1}[\,\widehat{t}\,]$ contains a copy of
  $\mathcal{H}_3$, and
  $t$ is in  $0^{<\om}$.
\end{lem}

\begin{proof}
Suppose
 $\widehat{t}$ is $0$-large.
 Then
  there is some $s\in 0^{<\om}$ such that
$\varphi^{-1}[\,\widehat{t}\,]\cap\widehat{s}$ is cofinal in $\widehat{s}$.
Since $s$ is in $0^{<\om}$,  the set of coding nodes in
$\widehat{s}$ represents a graph which contains a copy of $\mathcal{H}_3$.
In particular,
$\varphi^{-1}[\,\widehat{t}\,]\cap\widehat{s}$ being a collection of coding nodes which is  cofinal in $\widehat{s}$
implies that this set contains
coding nodes
representing a copy of $\mathcal{H}_3$.
Since $\varphi$ is passing number preserving, it follows that $\widehat{t}$ contains a copy of $\mathcal{H}_3$.
This would be impossible if $t$ were not a sequence of $0$'s.
Therefore, $t\in 0^{<\om}$.
\end{proof}

\begin{lem}\label{lem.Lemma1}
If $L\sse\overline{D}$ is $s$-large  and $L=\bigcup_{i<n}L_i$ is a  partition of $L$ into finitely many pieces,
then there is an $i<n$ such that $L_i$ is $s$-$0$-large.
\end{lem}

\begin{proof}
Suppose that no $L_i$ is $s$-$0$-large.
Then there is an $s_0$ which $0$-extends $s$
such that
$\varphi^{-1}[L_0]\cap \widehat{s_0}=\emptyset$.
Given $i<n-1$ and $s_i$, a $0$-extension of $s$, since $L_{i+1}$ is not $s$-$0$-large, there is some $s_{i+1}$ which $0$-extends $s_i$ such that
$\varphi^{-1}[L_{i+1}]\cap \widehat{s_{i+1}}=\emptyset$.
In the end,
we obtain  $s_{n-1}\in\bS$
which $0$-extends $s$
such that for all $i<n$,
$\varphi^{-1}[L_{i}]\cap \widehat{s_{n-1}}=\emptyset$.
This contradicts that
$L$ is $s$-large.
\end{proof}

\begin{lem}\label{lem.Lemma2}
Suppose  $t$ is in $\overline{D}\re |c^D_m|$
and  $\widehat{t}$  is    $s$-large.
Let  $n\ge m$ be given satisfying   $|c_n|>|s|$,
    and
    let $\ell\ge |c^D_n|+1$ be given.
 For $i<2$, let
\begin{equation}
J_i=
\bigcup\{\widehat{u}:u\in \widehat{t}\re \ell\mathrm{\ and\ }
u(|c^D_{n}|)=i\}.
\end{equation}
Then  for each $i<2$,
$J_{i}$ is large.
Moreover, $J_0$ is $s$-$0$-large.
\end{lem}

\begin{proof}
Let $i<2$ be fixed, and
suppose towards a contradiction that
$J_i$ is not  large.
Fix any $s_0\contains s$
satisfying  $|s_0|>|c_n|$, $s_0(|c_n|)=i$, and
$|\bar{\varphi}(s_0)|>\ell$.
Since $J_i$ is not large,
there is some $s_1\contains s_0$ such that
 $\varphi^{-1}[J_i]\cap \widehat{s_1}=\emptyset$.
 Fix some coding node $c_k\in
 \varphi^{-1}[\,\widehat{t}\,]\cap \widehat{s_1}$.
 Such a coding node exists since
 $ \varphi^{-1}[\,\widehat{t}\,]$ is a  cofinal subset of
 $\bC\cap\widehat{s_1}$.
Note that
$c_k\contains s_0$ implies  $c_k(|c_n|)=i$,
and
$c_k\in \varphi^{-1}[\,\widehat{t}\,]$ implies that $c^D_k=\varphi(c_k)\contains t$.
 Therefore, $c^D_k$ is a member of $J_i$.
 Hence, $c_k$ is in $\varphi^{-1}[J_i]\cap \widehat{s_1}$,
 contradicting that this set is empty.
Thus, $J_i$ must be large.

Now suppose that $J_0$ is not $s$-$0$-large.
Similar to the above argument,
take any $s_0$  which $0$-extends $s$
such that   $|s_0|>|c_n|$, $s_0(|c_n|)=0$, and
$|\bar{\varphi}(s_0)|>\ell$.
 Since $J_0$ is not $s$-$0$-large,
there is some $0$-extension  $s_1$ of $s_0$
 such that
 $\varphi^{-1}[J_i]\cap \widehat{s}_1=\emptyset$.
 Now take some
 $c_k\in
 \varphi^{-1}[\,\widehat{t}\,]\cap \widehat{s}_1$.
This time,
 $c_k(|c_n|)=0$, since
$c_k\contains s_0$,
and  again,
$c_k\in \varphi^{-1}[\,\widehat{t}\,]$ implies that $c^D_k=\varphi(c_k)\contains t$.
 Therefore, $c^D_k$ is a member of $J_0$, a contradiction.
Thus, $J_0$ must be  $s$-$0$-large.
\end{proof}

\begin{lem}\label{lem.Fact}
Suppose that $\widehat{t}$ is $s$-$0$-large and
$|t|<|\bar{\varphi}(s)|$.
Then for each $\ell>|\bar{\varphi}(s)|$,
there is an extension $u\contains t$  in $\widehat{D}$ with $|u|=\ell$ such that $\widehat{u}$ is $s$-$0$-large.
\end{lem}

\begin{proof}
Since $\widehat{t}$ is $s$-$0$-large, there is some
$0$-extension  $s_0$ of $s$  such that $\widehat{t}$ is    $s_0$-large. Let $n$ be any index such that $\ell>|c^D_n|>|t|$.
Letting $L_=\{\widehat{u}:u\in\widehat{t}\re\ell\mathrm{\ and\ } u(|c^D_n|)=0\}$,
 Lemma \ref{lem.Lemma2} implies there is some $u\in \widehat{t}\re\ell$ such that
 $\widehat{u}$ is $s_0$-$0$-large.
 Since $s_0$ is a $0$-extension of $t$,
 $\widehat{u}$ is  again $s$-$0$-large.
\end{proof}

We shall say that a pair of nodes $s,t$  is {\em unlinked} if it is not linked; that is, if there is no $\ell$ such that $s(\ell)=t(\ell)=1$.

\begin{lem}\label{lem.LemmaA}
Suppose that $t_0,t_1\in \overline{D}$ are unlinked,
and assume also that
 $|t_0|=|t_1|$.
 Given
 $s_0,s_1\in\bS$  of the same length
such that for each $i<2$,
$\widehat{t_i}$ is $s_i$-$0$-large,
then
\begin{enumerate}
\item[(a)]
$s_0$ and $s_1$ are unlinked; and
\item[(b)]
For each $\ell>|t_i|$ there are  $u_i\in\widehat{t_i}\re\ell$ such that $u_0$ and $u_1$ are unlinked, and
there are   $0$-extensions $x_i\contains s_i$ such that $\widehat{u_i}$ is $x_i$-large.
It follows that
 $x_0$ and $x_1$ are unlinked.
\end{enumerate}
\end{lem}

\begin{proof}
Let $j$ be the integer such that $|s_0|=|s_1|=|c_j|$, and let $k$ be the integer such that $|t_0|=|t_1|=|c^D_k|$.
By Lemma \ref{lem.Fact}, we may assume that
 each $|t_i|\ge|\bar{\varphi}(s_i)|$,  and hence,  $j\le k$.
Let
$\ell>|c^D_k|$ be given.
Since for each $i<2$, $\widehat{t_i}$ is $s_i$-$0$-large,
it follows that   $\bigcup\{\widehat{u}:u\in\widehat{t_i}\re\ell\}$ is also $s_i$-$0$-large.
By Lemma \ref{lem.Lemma1}, we can
fix  some $u_i\in\widehat{t_i}\re\ell$ such that
$\widehat{u_i}$ is $s_i$-$0$-large.
Thus, there is an $x_i$ which  $0$-extends $s_i$ such that
$\widehat{u_i}$ is $x_i$-large.
By $0$-extending one of the $x_i$'s if necessary, we may assume that $|x_0|=|x_1|$.
Take  coding nodes $c_{n_i}\in\varphi^{-1}[\widehat{u_i}]\cap\widehat{x_i}$;
without loss of generality, say $n_0<n_1$.

Since  each $c^D_{n_i}\contains t_i$ and the pair $t_0,t_1$ is unlinked,
  for each $m<k$,
  at least one of
$c_{n_0}^D(|c^D_m|)$ and $c_{n_i}^D(|c^D_m|)$  equals zero.
Then since $\varphi$ is passing number preserving and $j\le k$,
we have that $c_{n_0}$ and $c_{n_1}$ are unlinked below $|c_j|$.
Since  for each $i<2$, $c_{n_i}$ extends $s_i$, it follows that $s_0$ and $s_1$ are unlinked.
(This uses the fact that every level of $\bS$ has a coding node.)
Thus, (a) holds.

To  finish proving (b),
since $x_0,x_1$ are unlinked at any $m$ in the interval $[j,|x_1|)$, and since by (a), they are unlinked at any $m<j$, it follows that $x_0$ and $x_1$ are unlinked.
Furthermore, $\varphi^{-1}[\widehat{u_i}]\cap\widehat{x_i}$
is cofinal in $\widehat{x_i}$.
Therefore, we can choose  the coding nodes $c_{n_i}\contains x_i$ to have the additional property that  $c_{n_1}(|c_{n_0}|)=1$.
Since $\varphi$ is passing number preserving,
it also holds that
$c^D_{n_1}(|c^D_{n_0}|)=1$.
Since $c^D_{n_i}=\varphi(c^D_{n_i})\contains u_i$,
it must be the case that $u_0$ and $u_1$ are unlinked.
Hence, (b) holds.
\end{proof}

\begin{lem}\label{lem.LemmaAPart2}
Suppose that
\begin{enumerate}
\item
 $t_0,t_1\in \overline{D}$ are of the same length and
are unlinked.
\item
 $s_0,s_1\in\bS$  are  of the same length.
\item
$\widehat{t_0}$ is $s_0$-$0$-large.
\item
$u_1\contains t_1$  and satisfies
$\widehat{u_1}$ is $x_1$-large, for some $x_1\contains s_1$.
\end{enumerate}

Then
\begin{enumerate}
\item[(a)]
$s_0$ and $s_1$ are unlinked; and
\item[(b)]
There is some   $u_0\in\widehat{t_0}\re |u_1|$
and a   $0$-extension $x_0\contains s_0$
such that
$\widehat{u_0}$ is $x_0$-large,
 $u_0$
 and $u_1$ are unlinked, and $x_0$ and $x_1$ are unlinked.
\end{enumerate}
\end{lem}

\begin{proof}
Let $j$ be the integer such that $|s_0|=|s_1|=|c_j|$, and let $k$ be the integer such that $|t_0|=|t_1|=|c^D_k|$.
Then   $|t_0|\ge|\bar{\varphi}(s_0)|$ implies that $j\le k$.
Since $\widehat{t_0}$ is $s_0$-$0$-large,
letting
$\ell=|u_1|$,
it follows that   $\bigcup\{\widehat{u}:u\in\widehat{t_0}\re\ell\}$ is also $s_0$-$0$-large.
By Lemma \ref{lem.Lemma1}, we can
fix  some $u_0\in\widehat{t_0}\re\ell$ such that
$\widehat{u_0}$ is $s_0$-$0$-large.
Thus, there is an $x_0$ which  $0$-extends $s_0$ such that
$\widehat{u_0}$ is $x_0$-large.
By $0$-extending one of the $x_i$'s if necessary, we may assume that $|x_0|=|x_1|$.
Take  coding nodes $c_{n_i}\in\varphi^{-1}[\widehat{u_i}]\cap\widehat{x_i}$;
without loss of generality, say $n_0<n_1$.

Since  each $c^D_{n_i}\contains t_i$ and the pair $t_0,t_1$ is unlinked,
  for each $m<k$,
  at least one of
$c_{n_0}^D(|c^D_m|)$ and $c_{n_i}^D(|c^D_m|)$  equals zero.
Then since $\varphi$ is passing number preserving and $j\le k$,
we have that $c_{n_0}$ and $c_{n_1}$ are unlinked below $|c_j|$.
Since each $c_{n_i}$ extends $s_i$, it follows that $s_0$ and $s_1$ are unlinked.
(This uses the fact that every level of $\bS$ has a coding node.)
Thus, (a) holds.

We took $u_0\in\widehat{t_0}$  and $x_0$ to be a $0$-extension of $s_0$ so that  $\widehat{u_0}$ is $x_0$-large.
So to
 finish proving (b), we just need to show that $u_0$ and $u_1$ are unlinked.
 It suffices to show that there are coding nodes $c^D_{n_i}\contains u_i$ with $|c^D_{n_0}|<|c^D_{n_1}|$ and
 $c^D_{n_1}(|c^D_{n_0}|)=1$.
 Since $x_0$ is a $0$-extension of $s_0$, and since  by part (a), $s_0$ and $s_1$ are unlinked,
 it follows that $x_0$ and $x_1$ are unlinked.
Furthermore,  each $\varphi^{-1}[\widehat{u_i}]\cap\widehat{x_i}$
is cofinal in $\widehat{x_i}$.
Therefore, we can choose  the coding nodes $c_{n_i}\contains x_i$ to have the additional property that
$c_{n_1}(|c_{n_0}|)=1$.
Since $\varphi$ is passing number preserving,
it also holds that
$c^D_{n_1}(|c^D_{n_0}|)=1$.
Since $c^D_{n_i}=\varphi(c^D_{n_i})\contains u_i$,
it must be the case that $u_0$ and $u_1$ are unlinked.
By the same reasoning as for $s_0$ and $s_1$, it follows that  $x_0$ and $x_1$ are unlinked.
Hence, (b) holds.
\end{proof}

The next lemma  follows from
 Lemma \ref{lem.LemmaA} and
the  fact that $\bD$ is  canonically linked.

\begin{lem}\label{lem.LemmaB}
Suppose
 $X=\{s_i:i<p\}$ is a level set in $\bS$  and
$Y=\{t_i:i<p\}$ is a level set in $\overline{D}$  such that
\begin{enumerate}
\item
For each $i<p$,
$\widehat{t_i}$ is $s_i$-$0$-large.
\item
For each $i<j<p$, $t_i$ and $t_j$ are linked if and only if $s_i$ and $s_j$ are linked.
\end{enumerate}
Then for each $\ell>|t_0|$,
for each $i<p$ there is some $u_i\contains t_i$ of length $\ell$ and there is some $x_i$  which $0$-extends $s_i$
such that
each $\widehat{u_i}$ is $x_i$-large,
and  all $x_i$ have the same length.
Moreover, the set $\{u_i:i<p\}$ has no new  linked pairs over $Y$.
\end{lem}

\begin{proof}
Let  $X=\{s_i:i<p\}$ be  a level set in $\bS$
and let $Y=\{t_i:i<p\}$ be a level set in $\overline{D}$  with $|\bar{\varphi}(s_0)|\le |t_0|$
satisfying assumptions (1) and (2).
Let $j\le k$ be given such that each $|s_i|=|c_j|$ and each
$|t_i|=|c^D_k|$.
Since $\bD$ is canonically linked and $Y$ is a subset of $\bD$,
it follows that  $Y$ is canonically linked.

By Lemma \ref{lem.LemmaA},
there are $x_i$ $0$-extending $s_i$, all of the same length,
and there are $u_i\contains t_i$ all of length $\ell$ such that
each $\widehat{u_i}$ is $x_i$-large and each pair
$\{u_i,u_j\}$ is linked only if $\{t_i,t_j\}$ is linked.
Thus,  the set $Y'=\{u_i:i<p\}$ has no new  linked sets over $Y$.
(This follows from $\bD$ being incrementally linked:
For $Y'$ has no new linked sets  over $Y$ if and only if
 $Y'$ has no new linked pairs.)
Since $X'=\{x_i:i<p\}$ is a level set of $0$-extensions of $X$, it has no new  linked sets  over $X$.
\end{proof}

For level sets $X$ and $Y$ with the same cardinality,
we  say that $X$ and $Y$ have the {\em same linked pairs}    if and only if for all $i<j<p$,
$\{s_i,s_j\}$ is linked iff $\{t_i,t_j\}$ is linked,
where $\lgl s_i:i<p\rgl$ and $\lgl t_i:i<p\rgl$ are the lexicographically increasing enumerations of $X$ and $Y$, respectively.

Let $A$ be an antichain of coding nodes in $S$.
Let $W_A$ be a  minimal subset  of  $W$ such that each new essential linked pair in $A$ is witnessed by a coding node in $W_A$,  and let $B$ denote the meet-closure of $A\cup W_A$.
By our construction of $W$, we may assume that
each new essential linked pair $\{s,t\}$ in $A$ is witnessed by the coding node  $c$ of least length in $B$ above the minimal level  $\ell$ such that $s(\ell)=t(\ell)=1$.
Moreover, this coding node $c$ is the minimal critical node in $B$ above $\ell$ and forms no linked pair with any other member of $B$, and $A$ has no other new linked sets in the interval $[\ell,|c|]$.
($B$ can be thought of as a minimalistic kind of envelope for $A$.
The witnessing coding nodes in $W_A$ are best thought of as place holders  to keep track of levels where new essential linked pairs appear.)
Let
$\lgl b_i:i<N\rgl$, where $N\le\om$, enumerate the nodes in $B$ in order of increasing length.

We will be using the map $\bar{\varphi}$ to
construct an ep-similarity map $f$  of $B$ as a subset of $\bS$ into $\overline{D}$ as follows:
For $k>0$,
let $M_k=|b_{k-1}|+1$.
Define $\widehat{D}$ to be the
tree of all initial segments of members of $D$;
thus,
$\widehat{D}= \{u\re \ell:u\in D$ and $\ell\le |u|\}$.
For  $k<\om$
we will recursively define meet-closed sets $T_k\sse\widehat{D}$,
$N_k<\om$,
a level set $\{s_t:t\in T_k\re N_k\}\sse\bS$, and
 maps $f_k$ and $\psi_k$
 such that the following hold:
\begin{enumerate}
\item
$f_k$ is an ep-similarity embedding of $\{b_i:i<k\}$ onto a subset $\{t_i:i<k\}\sse T_k$.
\item
$N_k=\max\{|t|:\in T_k\}$.

\item
All maximal nodes of $T_k$ are either in $T_k\re N_k$, or else in the range of $f_k$.

\item
For each $t\in T_k\re N_k$,
$\widehat{t}$ is  $s_t$-large.

\item
$\psi_k$ is a $\prec$ and passing type preserving  bijection of
$B\re M_k$ to $T_k\re N_k$.
\item
$B\re M_k$,
$T_k\re N_k$,
 and $\{s_t:t\in T_k\re N_k\}$  all have the same linked  pairs;
that is, the pair
$\{y,z\}\sse B\re M_k$ is linked if and only if
$\{\psi_k(y),\psi_k(z)\}$ is linked
 if and only if
$\{s_{\psi_k(y)},s_{\psi_k(z)}\}$ is linked.

\item
$T_{k-1}\sse T_k$, $f_{k-1}\sse f_k$, and $N_{k-1}<N_k$.
\end{enumerate}

The idea behind $T_k$ is that it will contain an
ep-similarity  image of $\{b_i:i<k\}\cup (B\re M_k)$,
the nodes in the image of $B\re M_k$ being the ones we need to continue extending in order to build an ep-similarity from $B$ into $D$.

To begin, let $r$ denote the root of $\cl(D)$ and assume, without loss of generality, that $|r|\ge 1$.
Then $\widehat{r}=\overline{D}$, which is $\emptyset$-large, and $|r|<|c^D_0|$.
By Lemma \ref{lem.Lemma1},
there is a $t_{-1}\in D\re |c^D_0|$ such that
$\widehat{t_{-1}}$ is $0$-large.
By Lemma \ref{lem.Fact2},
$t_{-1}$ is in $0^{<\om}$.
Let $s_{-1}\in\bS$ be a  node of  minimum length such that
$\varphi^{-1}[\widehat{t_{-1}}]\cap\widehat{s_{-1}}$
is cofinal in  $\widehat{s_{-1}}$.
Note that $s_{-1}$  must be  in $0^{<\om}$.
Define $T_{-1}=\{t_{-1}\}$, $f_0=\emptyset$,  $N_{-1}=|t_{-1}|$ (which equals $|c^D_0|$), and
letting $b_{-1}=c_0$, let
$M_{-1}=0$ and
$\psi_{-1}(b_{-1})=t_{-1}$ (noting that $c_0$ is an initial segment of all nodes in $B$).

Assume  now that  $k\ge 0$ and   (1)--(7) hold for all $k'< k$.
We have two cases: Either $b_k$ is a splitting node or else it is a coding node.
\vskip.1in

\noindent\bf Case I. \rm
$b_k$ is a splitting node.
\vskip.1in

Let $t$ denote $\psi_k(b_k\re M_k)$.
Then
by (5),
 $t$ is a member of $T_k\re N_k$,
 and
by (4), $\widehat{t}$ is $s_t$-large.
Fix  a coding node $c_j\in \varphi^{-1}[\,\widehat{t}\,]\cap\bC$.
The purpose of
 $\varphi(c_j)$ (which we recall, is exactly $c^D_j$) is just to find a level of $\overline{D}$ where we
 can find two distinct nodes which extend $t$, so that $t$ will be extended to a splitting node.
For $i\in \{0,1\}$,
let  $\ell=|c^D_{j+1}|$, and let
\begin{equation}
J_i=\bigcup\{\widehat{u}:u\in\widehat{t}\re \ell\mathrm{\  and \ }
u(|\varphi(c_j)|)=i\}.
\end{equation}
Then
by Lemma \ref{lem.Lemma2},
both
$J_i$  are large, and moreover, $J_0$ is ${s_t}$-$0$-large.
Therefore, by Lemma \ref{lem.Lemma1},
for $i\in\{0,1\}$,
 there is $u_i\in \widehat{t}\re \ell$ and  an extension $s_i$
 of $s_t$ such that
$\widehat{u_i}$ is $s_i$-large;
moreover,  this lemma ensures that we may take $s_0$ to be a $0$-extension of $s_t$.
If $s_t$ is in $0^{<\om}$, then so also $s_0$ is in $0^{<\om}$.
Note that $u_0$ and $u_1$ are incomparable, since they have different passing numbers at $\varphi(c_j)$.
Moreover,  their meet,
$u_0\wedge u_1$, extends $t$.
Define $N_{k+1}=\ell$, which is $|u_i|$ for both $i\in\{0,1\}$.

For all other $y\in (T_k\re N_k)\setminus \{t\}$,
by (4) we know that $\widehat{t}$ is $s_y$-large.
Again by Lemma \ref{lem.Lemma1},
there is a $y'\in \widehat{y}\re N_{k+1}$ such that
$y'$ is $s_y$-$0$-large.
Therefore, there is a $0$-extension  $s_{y'}\contains s_y$ such that
$\widehat{y'}$ is $s_{y'}$-large.
Note  that if $s_y$ is in $0^{<\om}$, then so also $s_{y'}$ is in $0^{<\om}$.
By $0$-extending some of  the nodes if necessary, we may assume that $s_0$, $s_1$  and all $s_{y'}$ have the same length.

Define
$t_k=u_0\wedge u_1$, and extend the map $f_k$ by letting  $f_{k+1}(b_k)=t_k$.
Then $f_k$ is an essential pair strict similarity embedding of $\{b_i:i\le k\}$ onto $\{t_i:i\le k\}$ so (1) holds.
Define
\begin{equation}
T_{k+1}=T_k\cup
\{u_0,u_1, d_k\}
\cup\{y':y\in (T\re N_k)\setminus\{t\}\}.
\end{equation}
Letting  $s_{u_i}$ denote $s_i$ for $i\in\{0,1\}$,
we have  the level set
$\{s_z:z\in T_{k+1}\re N_{k+1}\}$ with the following properties:
For each $z\in  T_{k+1}\re N_{k+1}$,
$\widehat{z}$ is $s_z$-large, so (4) holds.
By Lemma \ref{lem.LemmaAPart2}, this set
$\{s_z:z\in T_{k+1}\re N_{k+1}\}$
 and $T_{k+1}\re N_{k+1}$ have
the same linked pairs, precisely because they each have no new linked pairs.
Since $b_k$ is a splitting node,
$B\re M_{k+1}$ has no new linked pairs over $B\re M_k$.
Thus,
 (5) holds.

Define $\psi_{k+1}$
on  $\bD\cap (\widehat{b_k}\re M_{k+1})$
to be the unique $\prec$-preserving map
onto $\{u_0,u_1\}$; then (6) holds.
Properties (2), (3), and (7) hold by our construction.
This completes Case I.
\vskip.1in

\noindent\bf Case II. \rm
$b_k$ is a coding  node.
\vskip.1in
Let
$t_*$ denote $\psi_k(b_k\re  M_k)\in T_k\re N_k$.
Choose a coding node $c_p\contains s_{t_*}$ in $\bS$ such that
 $|c_p| >|s_t|$ for all $t\in T_k\re N_k$,
 and $|\varphi(c_p)|>t$ for all $t\in T_k\re N_k$.
 Define $t_k=\varphi(c_p)$.
Extend $f_k$ by defining
 $f_{k+1}(b_k)=t_k$,
 and
let $N_{k+1}=|c^D_{p+1}|$.
For   each $i\in\{k,k+1\}$,
let
$E_i=(A\re M_i)\setminus \{b_i\re M_i\}$.
Note that  for each $s\in E_k$, there is a unique $e'\in  E_{k+1}$
such that $e'\contains e$.

Fix any  $e\in E_k$ and let $t:=\psi_k(e)\in T_k\re N_k$.
Let $i$ denote the passing number of $e'$ at $b_k$.
By Lemma \ref{lem.Lemma2},
\begin{equation}
J_i:=
\bigcup\{\widehat{u}:u\in\widehat{t}\re N_{k+1}\mathrm{\ and\ }
u(|t_k|)=i\}
\end{equation}
is large, and is $s_t$-$0$-large if $i=0$.
Thus, by Lemma \ref{lem.Lemma1},
there is some $t'\in\widehat{t}\re N_{k+1}$ such that $\widehat{t'}$ is
large, and is $s_t$-$0$-large if $i=0$.
If $i=0$,
let $s_{t'}$ be a $0$-extension of $s_t$  of length long enough that $\widehat{t'}$ is $s_{t'}$-large.
If $i=1$, let $s_{t'}$ be an extension of $s_t$ such that $\widehat{t'}$ is $s_{t'}$-large.
Thus, (4) will hold, given our definition of $T_{k+1}$ below.
Define $\psi_{k+1}(e')=t'$, and note that $\psi_{k+1}(e')(|t_k|)= e'(|b_k|)$.
As $\psi_{k+1}$ is $\prec$-preserving,  (5) holds.

Let
\begin{equation}
T_{k+1}=
T_k\cup\{b_k\}\cup\{\psi_{k+1}(e'):e'\in E_{k+1}\}.
\end{equation}
By $0$-extending some of the nodes $s_{t'}$ if necessary, we may assume that
the nodes
in the set $\{s_{t'}:  t'\in T_{k+1}\re N_{k+1}\}$ all
  have the same length.
By
the induction hypothesis (6) for $T_k\re N_k$ and $\{s_t:t\in T_k\re N_k\}$
and Lemmas
\ref{lem.LemmaA}
and
\ref{lem.LemmaAPart2},
for each pair $t,u\in T_k\re N_k$,
the pair of nodes $s_{t'},s_{u'}$  is linked if and only
$t',u'$ are linked.
Rewriting,
we have that
$T_{k+1}\re N_{k+1}$
and $\{s_t:t\in T_{k+1}\re N_{k+1}\}$ have the same linked pairs.
Furthermore,  any pair $t',u'\in  T_{k+1}\re N_{k+1}$ are linked only if their $\psi_{k+1}$-preimages in $B\re M_{k+1}$ are linked.
Thus, (6) holds.
By our construction,
(1), (2), (3), and (7) hold  as well.

This concludes the construction in Case II.
\vskip.1in

%***
%(Do we need to write the following?)
%***
%By the induction hypothesis, it holds that
%\begin{align}
%v'(f_k(c^A_i);f_k[A\re i])
%& =
%v(f_k(c^A_i);f_k[A\re i])\cr
%&\sim
%u(c_i^A;A\re i)\cr
%&=u'(c_i^A;A\re i)
%\end{align}
%for each $i<j$.
%***
%By (5) in the induction hypothesis,
%$\psi_k$ is $\prec$ and passing type preserving;
%hence,
%$\psi_k(u)(f_k(c^A_i);f_k[A\re i])\sim
%u(c^A_i;A\re i)$,
%for each $i<j$.

Finally, let $f=\bigcup_k f_k$.
Then $f$ is a strong similarity map from $B$ to $f[B]$:
$f$ preserves the $\prec$-order and the splitting and coding node order of $B$.
Moreover, whenever $b_k$ is a coding node, the construction of $\psi_{k+1}$ ensures that for any $k<n<N$  such that $b_n$ is a coding node in $B$,
$f_{n+1}(b_n)$ has the same passing number at $f_{k+1}(b_k)$ as the passing number of $b_n$ at $b_k$.
So $f$ is passing type preserving.
Further,  property (6) along with the properties of $\bD$ ensures that
 $f$ is an ep-similarity map from $B$ to $f[B]$.
 It follows that $f\re A$ is an ep-similarity from $A$ to $f[A]$.
\end{proof}

\begin{thm}[Canonical Partitions]\label{thm.canonpart}
Let $\bG$ be a finite triangle-free graph and let $h$ be a coloring of all copies of $\bG$ inside $\mathcal{H}_3$.
Then there is a subgraph $\mathcal{H}$ of $\mathcal{H}_3$ which isomorphic to $\mathcal{H}_3$
in which  for each $A\in\Simp(\bG)$,
all copies of $\bG$ in $\Simp(A)$ have the same color.
Moreover,  each $A\in\Simp(\bG)$ persists in the coding tree induced by $\mathcal{H}$.
\end{thm}

\begin{proof}
Let  $\bG\in\mathcal{G}_3$  be given,
and  fix a finite coloring $h$ of the copies of $\bG$ in $\mathcal{H}_3$.
Let $\bS$ be a coding tree representing $\mathcal{H}_3$.
By the Partition Theorem
\ref{thm.partitionthm},
there is  an antichain $\bD$ of
a canonically linked subtree $S\sse \bS$ such that all
sets of coding nodes $A$ in $\bD$ representing $\bG$ with the same essential pair strict similarity type
have the same color.

By Theorem \ref{thm.persistence},
every ep-similarity type
in $\Simp(\bG)$
 persists in the coding tree by
any  isomorphic subgraph of $\mathcal{H}_3$.
Therefore, $\{\Simp(A):A\in\Simp(\bG)\}$ is a canonical partition.
Hence the big Ramsey degree $T(\bG,\mathcal{H}_3)$ is exactly the  cardinality of the set
$\Simp(\bG)$.
\end{proof}

\begin{cor}\label{cor.main}
Given a finite triangle-free graph $\bG$,
the big Ramsey degree of $\bG$ in the triangle-free graph is exactly  the number of essential pair similarity types of strongly skew antichains coding $\bG$:
\begin{equation}
T(\bG,\mathcal{H}_3)=|\Simp(\bG)|.
\end{equation}
\end{cor}

%***We're going to want to mention the ordered version, and show some graph way of viewing this.***

%%%%%%%%%%%%%%%%%%%%%
%%%%%%%%%%%%%%%%%%%%%
%%%%%%%%%%%%%%%%%%%%%
%%%%%%%%%%%%%%%%%%%%%

%\section{The triangle-free Henson graph has a big Ramsey structure}

%As  canonical partitions satisfy Zucker's criterion for a big Ramsey structure, it follows that his results on universal completion flows in \cite{Zucker19} apply to the triangle-free Henson graph.

%Show here how Zucker's criteria is satisfied.

%%%%%%%%%%%%%%%%%%%%%%
%%%%%%%%%%%%%%%%%%%%%%
%%%%%%%%%%%%%%%%%%%%%%
%%%%%%%%%%%%%%%%%%%%%%
%%%%%%%%%%%%%%%%%%%%%%
%%%%%%%%%%%%%%%%%%%%%%

%\section{Concluding Remarks}\label{sec.infdiml}

\bibliographystyle{amsplain}
\bibliography{referencesH3LowerBds}

\end{document}